\documentclass[english, reqno]{amsart}
\usepackage{babel}
\usepackage{caption}
\usepackage{tikz}
\usepackage[latin1]{inputenc}
\usepackage{color}
\usepackage{hyperref}
\usepackage{epsfig}

\def\vint_#1{\mathchoice%
          {\mathop{\kern 0.2em\vrule width 0.6em height 0.69678ex depth -0.58065ex
                  \kern -0.8em \intop}\nolimits_{\kern -0.4em#1}}%
          {\mathop{\kern 0.1em\vrule width 0.5em height 0.69678ex depth -0.60387ex
                  \kern -0.6em \intop}\nolimits_{#1}}%
          {\mathop{\kern 0.1em\vrule width 0.5em height 0.69678ex
              depth -0.60387ex
                  \kern -0.6em \intop}\nolimits_{#1}}%
          {\mathop{\kern 0.1em\vrule width 0.5em height 0.69678ex depth -0.60387ex
                  \kern -0.6em \intop}\nolimits_{#1}}}
                  
                  \newcommand{\aveint}[2]{\mathchoice%
          {\mathop{\kern 0.2em\vrule width 0.6em height 0.69678ex depth -0.58065ex
                  \kern -0.8em \intop}\nolimits_{\kern -0.45em#1}^{#2}}%
          {\mathop{\kern 0.1em\vrule width 0.5em height 0.69678ex depth -0.60387ex
                  \kern -0.6em \intop}\nolimits_{#1}^{#2}}%
          {\mathop{\kern 0.1em\vrule width 0.5em height 0.69678ex depth -0.60387ex
                  \kern -0.6em \intop}\nolimits_{#1}^{#2}}%
          {\mathop{\kern 0.1em\vrule width 0.5em height 0.69678ex depth -0.60387ex
                  \kern -0.6em \intop}\nolimits_{#1}^{#2}}}

\newcommand{\R}{\mathbb{R}}
\newcommand{\N}{\mathbb{N}}
\newcommand{\Z}{\mathbb{Z}}
\newcommand{\Q}{\mathbb{Q}}
\newcommand{\F}{\mathcal{F}}
\newcommand{\M}{\mathcal{M}}
\newcommand{\E}{\mathbb{E}}
\renewcommand{\P}{\mathbb{P}}
\renewcommand{\L}{\mathcal{L}}
\renewcommand{\O}{\mathcal{O}}

\newcommand{\Om}{\Omega}
\newcommand{\eps}{\varepsilon}

\newcommand{\ol}{\overline}

\DeclareMathOperator*{\diam}{diam} 
\DeclareMathOperator*{\osc}{osc} 
\newcommand{\dist}{\operatorname{dist}}

\newcommand{\abs}[1]{\left| #1 \right|}
\newcommand{\pare}[1]{\left(#1\right)}
\newcommand{\set}[2]{\{#1 \; :\; #2\}}
\newcommand{\prodin}[2]{\langle #1,#2 \rangle}

\def\1{\raisebox{2pt}{\rm{$\chi$}}}

\theoremstyle{plain}
\newtheorem{definition}{Definition}[section]
\newtheorem{proposition}[definition]{Proposition}
\newtheorem{theorem}[definition]{Theorem}
\newtheorem{corollary}[definition]{Corollary}
\newtheorem{lemma}[definition]{Lemma}
\newtheorem{remark}[definition]{Remark}
\newtheorem{example}[definition]{Example}

\theoremstyle{definition}
\theoremstyle{remark}

\numberwithin{equation}{section}

\begin{document}


\title[H\"older regularity with bounded and measurable increments]{H\"older regularity for stochastic processes with bounded and measurable increments}

\author[Arroyo]{\'Angel Arroyo}
\address{Departamento de An\'alisis Matem\'atico y Matem\'atica Aplicada, Universidad Complutense de Madrid, 28040 Madrid, Spain}
\email{ar.arroyo@ucm.es}

\author[Blanc]{Pablo Blanc}
\address{Department of Mathematics and Statistics, University of Jyv\"askyl\"a, PO~Box~35, FI-40014 Jyv\"askyl\"a, Finland}
\email{pablo.p.blanc@jyu.fi}

\author[Parviainen]{Mikko Parviainen}
\address{Department of Mathematics and Statistics, University of Jyv\"askyl\"a, PO~Box~35, FI-40014 Jyv\"askyl\"a, Finland}
\email{mikko.j.parviainen@jyu.fi}

\date{\today}


\maketitle

\begin{abstract}

We obtain an asymptotic H\"older estimate for expectations of a quite general class of discrete stochastic processes. Such expectations can also be described as solutions to a dynamic programming principle or as solutions to discretized PDEs. The result, which is also generalized to functions satisfying Pucci-type inequalities for discrete extremal operators, is a counterpart to the Krylov-Safonov regularity result in PDEs. However,  the discrete step size $\eps$ has some crucial effects compared to the PDE setting. The proof combines analytic and probabilistic arguments.
\end{abstract}

\section{Introduction}

The celebrated Krylov-Safonov \cite{krylovs79} H\"older estimate is one of the key results in the theory of non-divergence form elliptic partial differential equations with  bounded and measurable coefficients.  The result, in addition to being important on its own right, also gives a flexible tool in the higher regularity and existence theory due to its very general assumptions on the coefficients.

In this paper,  we study regularity of expectation of a quite general class of discrete stochastic processes or equivalently functions satisfying the dynamic programming principle  (DPP)  
\begin{align}
\label{eq:dpp-intro}
u (x) =\alpha  \int_{\R^N} u(x+\eps z) \,d\nu_x(z)+\frac{\beta}{\abs{B_\eps}}\int_{B_\eps(x)} u(y)\,dy
+\eps^2 f(x),
\end{align}	
where $f$ is a Borel measurable bounded function, $\nu_x$ is a symmetric probability measure for each $x$ with support in $B_\Lambda,\ \Lambda\geq 1$, $\eps>0$, and $\alpha+\beta=1,\ \alpha\ge 0,\ \beta>0$. 
From a stochastic point of view, our processes are generalizations of the random walk where the next step in the process is taken according to a probability measure that is a combination of $\nu_x$ and the uniform distribution on $B_\eps(x)$ as described by the DPP (more details are given in Sections~\ref{sec:dyna} and~\ref{sec:stoc}). 

It is important to notice that $\nu_x$ can vary quite freely from point to point. Under continuity  or other assumptions not satisfied in our case, related results have been studied for example  in  \cite{lindvallr86}, \cite{caffarellis09} and \cite{kusuoka15}.

The first of the two main results of this article is a H\"older estimate in the discrete setup without any further continuity assumption on the measures $\nu_x$.
\begin{theorem}
\label{Holder}
There exists $\eps_0>0$ such that if $u$ is a function satisfying \eqref{eq:dpp-intro} in $B_{2R}$ with $\eps<\eps_0R$, then
\[
|u(x)-u(z)|\leq \frac{C}{R^\gamma}\left(\sup_{B_{2R}}|u|+R^2\|f\|_\infty\right)\Big(|x-z|^\gamma+\eps^\gamma\Big)
\]
for every $x, z\in B_R$, where $C>0$ and $\gamma>0$ are constants independent of $\eps$.
\end{theorem}

The role of the discrete processes we study can be compared to the role of linear uniformly elliptic partial differential equations with bounded and measurable coefficients in the theory of PDEs.

The regularity techniques in PDEs, see \cite{krylovs79} and \cite{caffarellic95}, or in the nonlocal setting, see \cite{caffarellis09} and \cite{guillens12} utilize, heuristically speaking, the fact that there is information available in all scales.
We also refer to \cite{changlarad12} and \cite{schwabs16} for similar results regarding nonlocal operators with nonsymmetric kernels.
For a discrete process, the step size sets a natural limit for the scale, and this limitation has some crucial effects. Indeed, the value can even be discontinuous, and our estimates are asymptotic. Such estimates suffice in many applications, for example in passing to the limit with solutions to discretized PDEs or stochastic processes.
Similar results have been obtained on a grid in the context of difference equations with random coefficients in \cite{kuot90}. See also \cite{lawler91}, where regularity estimates for difference equations arising from random walks are obtained using probabilistic techniques.

The proof uses a stochastic approach akin to the original proof of Krylov and Safonov in \cite{krylovs79} with some crucial differences. The first observation, as suggested above, is that the function $u$ in \eqref{eq:dpp-intro} can be presented as an expectation. The key estimate is then Theorem~\ref{thm:main} stating that we can reach any set of positive measure with a positive probability before exiting a bigger cube. With this result at our disposal, the De Giorgi oscillation estimate, Lemma~\ref{DeGiorgi}, follows in a straightforward manner. Indeed, we can reach a level set with a positive probability and use this in estimating the oscillation. The H\"older estimate,  Theorem~\ref{Holder}, then follows by  De Giorgi oscillation lemma after a finite iteration.

The proof of Theorem~\ref{thm:main} is nonstandard. In the proof we would like to construct a set of cubes which is large enough and such that the set we want to reach has a high enough density in the cubes.  Both conditions, however, cannot always be satisfied simultaneously in our setting. As suggested above, both the PDE and nonlocal techniques utilize the information in all scales. Concretely, a rescaling argument is used in those proofs in arbitrary small cubes. In contrast, in our case the step size $\eps$ determines limit for the scale. If we simply drop all the cubes smaller than of scale $\eps$ in the usual Calder\'on-Zygmund decomposition, we have no control on the size of the error. Therefore, the cubes of scale $\eps$ need to be taken into account separately both in the decomposition lemma, Lemma~\ref{CZ}, as well as in the proof of the key intermediate result, Theorem~\ref{thm:main}. 

The proof of Theorem~\ref{thm:main} is based on the $\eps$-version of the Alexandrov-Bakelman-Pucci (ABP) estimate with bounded and measurable right hand side, Theorem~\ref{boundomega}.  However, the classical proof of the ABP estimate using the change of variables formula for integrals to obtain a quantity that can be estimated by the PDE  does not seem directly applicable here. Instead, we adapt the nonlocal approach of Caffarelli and Silvestre \cite{caffarellis09}. A second remark is that we directly apply the ABP estimate with a discontinuous right hand side, which is chosen to be a characteristic function of a level set. In this case, the standard ABP estimate having the $L^N$-norm on the right hand side is false (Example~\ref{ex:LN-abp-false}), and therefore the statement of Theorem~\ref{boundomega} is weaker, but sufficient for our purposes.

 Our study is partly motivated by the aim of developing stochastic methods in connection with the $p$-Laplace equation and other nonlinear PDEs, see Example~\ref{ex:plaplace} and Remark~\ref{aplications}. 
The H\"older estimate, Theorem~\ref{Holder}, can be generalized to functions merely satisfying the Pucci-type inequalities,
\begin{eqnarray}
	&\displaystyle
	u(x)
	\leq
	\alpha\sup_{z\in B_\Lambda}\frac{u(x+\eps z)+u(x-\eps z)}{2}+\frac{\beta}{\abs{B_\eps}}\int_{B_\eps(x)} u(y)\,dy
	+\eps^2 |f(x)|,\label{eq:maximal-dpp}\\
	&\displaystyle
	u(x)
	\geq
	\alpha\inf_{z\in B_\Lambda}\frac{u(x+\eps z)+u(x-\eps z)}{2}+\frac{\beta}{\abs{B_\eps}}\int_{B_\eps(x)} u(y)\,dy
	-\eps^2 |f(x)|,\label{eq:minimal-dpp}
\end{eqnarray}
later modified and shortened to the form
\begin{align*}
\label{L+L-ineq}
\L_\eps^+ u\ge -\abs{f},\quad \L_\eps^- u\le  \abs{f}. 
\end{align*}
This is our second main result.
\begin{theorem}
\label{HolderL+L-(intro)}
There exists $\eps_0>0$ such that if $u$ is a function satisfying \eqref{eq:maximal-dpp} and \eqref{eq:minimal-dpp} for every $x\in B_{2R}$ with $\eps<\eps_0R$, then
\[
|u(x)-u(z)|\leq \frac{C}{R^\gamma}\left(\sup_{B_{2R}}|u|+R^2\|f\|_\infty\right)\Big(|x-z|^\gamma+\eps^\gamma\Big)
\]
for every $x, z\in B_R$, where $C>0$ and $\gamma>0$ are constants independent of $\eps$.
\end{theorem}

We refer the reader to Section~\ref{sec:pucci} and in particular to Theorem~\ref{HolderL+L-} for a more detailed description.

\section{Preliminaries}
\label{preliminaries}

As above, let  $\Lambda\geq 1$, $\eps>0$, $\beta\in (0,1]$ and $\alpha=1-\beta$.
Every constant may depend on $\Lambda$, $\alpha$, $\beta$ and the dimension $N$. 
If a constant depends on other parameters we denote it.

Throughout this paper $\Omega \subset \mathbb{R}^N$ denotes a bounded domain, and further $B_r(x)=\{y\in\R^N:|x-y|<r\}$ as well as $B_r=B_r(0)$. 
We construct an extended domain containing all balls $B_{\Lambda\eps}(x)$ with $x\in\Omega$ as follows
\begin{equation*}
		\widetilde\Omega_{\Lambda\varepsilon}
		:\,=
		\set{x\in\R^n}{\dist(x,\Omega)<\Lambda\varepsilon}.
\end{equation*}
We follow the convention
\[
\int u(x)\,dx=\int_{\R^N} u(x)\,dx
\quad
\text{ and }
\quad
\vint_A u(x)\,dx=\frac{1}{|A|}\int_{A} u(x)\,dx.
\]
Further,
\[
\|f\|_{L^N(\Omega)}=\left(\int_\Omega f(x)^N\,dx\right)^{1/N}
\]
and
\[
\|f\|_{L^\infty(\Omega)}=\sup_\Omega |f|.
\]
When no confusion arises we just simply denote $\|\cdot\|_N$ and $\|\cdot\|_\infty$, respectively.

For each $x=(x_1,\ldots,x_n)\in\R^N$ and $r>0$, we define $Q_r(x)$ the open cube of side-length $r$ and center $x$ with faces parallel to the coordinate hyperplanes. More precisely,
\begin{equation*}
	Q_r(x)
	:\,=
	\{y\in\R^N\,:\,|y_i-x_i|<r/2,\ i=1,\ldots,n\}.
\end{equation*}
In addition, if $Q=Q_r(x)$ and $\ell>0$, for simplicity we denote $\ell Q=Q_{\ell r}(x)$.

\subsection{Dynamic programming principle and difference operators}
\label{sec:dyna}

We consider $\M(B_\Lambda)$ the set of symmetric unit Radon measures with support in $B_\Lambda$ and $\nu:\R^N\to \M(B_\Lambda)$
such that
\begin{equation}\label{measurable-nu}
x\longmapsto\int u(x+z) \,d\nu_x(z)
\end{equation}
defines a Borel measurable function for every $u:\R^N\to \R$ Borel measurable.
Then for each $x\in\R^N$ we have a measure $\nu_x$ with support in $B_\Lambda$ such that
\begin{align*}
\nu_x(E)=\nu_x(-E)
\end{align*}
for every measurable set $E\subset\R^N$.

It is worth remarking that the hypothesis \eqref{measurable-nu} on Borel measurability holds, for example, when the $\nu_x$'s are the pushforward of a given probability measure $\mu$ in $\R^N$. More precisely, if there exists a Borel measurable function $h:\R^N\times \R^N\to B_\Lambda$ such that 
\[
\nu_x=h(x,\cdot)\#\mu
\]
for each $x$, then
\[
\begin{split}
v(x)
&=\int u(x+z) \,d\nu_x(z)\\
&=\int u(x+h(x,y)) \,d\mu(y)\\
\end{split}
\]
is measurable by Fubini's theorem.

For each $\eps>0$ we consider a generalized random walk starting at $x_0\in\Omega$.
Given the value of $x_k$, the next position of the process $x_{k+1}$ is determine as follows. 
A biased coin is tossed.
If we get heads (probability $\alpha$), a vector $z$ is chosen according to $\nu_{x_k}$ and we have $x_{k+1}= x_k+ \eps z$.
If we get tails (probability $\beta$), $x_{k+1}$ is distributed uniformly in the ball $B_{\eps}(x_k)$. More details are given in Section~\ref{sec:stoc}.
Denote by $\tau$ the exit time from the domain, that is 
$$\tau=\min\{n\in\N:x_n\not\in \Omega\}.$$
Given a Borel measurable bounded function $g :\mathbb{R}^N \setminus \Omega \to \mathbb{R}$, we can define
$$
u(x_0):\,=\mathbb{E}^{x_0} [g (x_\tau)],
$$
where $\mathbb{E}^{x_0}$ stands for the expectation with respect to the process.
We will prove in Section~\ref{sec:DPP} that $u :\mathbb{R}^N \to \mathbb{R}$ satisfies the homogenous dynamic programming principle given by
\begin{align*}
u (x) =\alpha  \int u(x+\eps z) \,d\nu_x(z)+\beta\vint_{B_\eps(x)} u(y)\,dy
\end{align*}
for $x \in \Omega$, and $u(x) = g(x)$ for $x \not\in \Omega$.
Moreover, $u$ is the unique bounded function that satisfies the dynamic programming principle.

Moreover, given the running payoff $f:\Omega \to \R$, a Borel measurable bounded function, we can define
\[
u(x):\,=\mathbb{E}^{x} \left[\eps^2\sum_{i=0}^{\tau-1}f(X_i)+g (X_\tau)\right].
\]
In Section~\ref{sec:DPP} we prove that $u$ is the unique bounded function that satisfies the dynamic programming principle (DPP)
\begin{equation}
\label{DPP}
u (x) =\alpha  \int u(x+\eps z) \,d\nu_x(z)+\beta\vint_{B_\eps(x)} u(y)\,dy+\eps^2 f(x)
\end{equation}
for $x \in \Omega$ and $u(x) = g(x)$ for $x \not\in \Omega$. For clarity, let us emphasize that \eqref{DPP} is what we call the DPP in this paper.  
This also motivates the following definitions.

\begin{definition}
\label{def:solutions}
We say that a bounded Borel measurable function $u$ is a subsolution to the DPP if it satisfies
\[
u (x)\leq\alpha  \int u(x+\eps z) \,d\nu_x(z)+\beta\vint_{B_\eps(x)} u(y)\,dy+\eps^2 f(x)
\]
in $\Omega$.
Analogously, we say that $u$ is a supersolution if the reverse inequality holds.
If the equality holds, we say that it is a solution to the DPP.
\end{definition}

If we rearrange the terms in the DPP,  we may alternatively use a notation that is closer to the difference methods.
\begin{definition}
Given a Borel measurable bounded function $u:\R^N\to \R$, we define $\L_\eps u:\R^N\to \R$ as
\[
\L_\eps u(x)=\frac{1}{\eps^2}\left(\alpha  \int u(x+\eps z) \,d\nu_x(z)+\beta\vint_{B_\eps(x)} u(y)\,dy-u(x)\right).
\]
With this notation,  $u$ is a subsolution (supersolution) if and only if $\L_\eps u+f \geq 0 (\leq 0)$.
\end{definition}

By defining $\delta u(x,y):\,=u(x+y)+u(x-y)-2u(x)$ and recalling the symmetry condition on $\nu_x$ we can rewrite
\begin{equation}\label{L-eps}
\L_\eps u(x)=\frac{1}{2\eps^2}\left(\alpha  \int \delta u(x,\eps z) \,d\nu_x(z)+\beta\vint_{B_1} \delta u(x,\eps y)\,dy\right).
\end{equation}

Our theorems actually hold for functions merely satisfying Pucci-type inequali-
ties. For expositional clarity we leave this to Section~\ref{sec:pucci}.
\subsection{Examples}

In this section, we present some recent examples and applications. The list is by no means exhaustive, and further examples could be obtained discretizing partial differential operators with bounded and measurable coefficients.

\begin{example}[Convergence to the solution of the PDE]
\label{ex:linear}
Let $\phi\in C^2(\Omega)$. We can use the second order Taylor's expansion of $\phi$ to obtain an asymptotic expansion for $\L_\eps\phi(x)$. Indeed, observe that
\begin{equation*}
	\delta\phi(x,\eps y)
	=
	\eps^2\mathrm{Tr}\{D^2\phi(x)\cdot y\otimes y\}+o(\eps^2)
\end{equation*}
holds as $\eps\to 0$ for every $y\in B_\Lambda$, where $a\otimes b$ stands for the tensor product of vectors $a,b\in\R^n$, that is, the matrix with entries $(a_ib_j)_{ij}$. Hence, by the linearity of the trace,
\begin{equation*}
\begin{split}
	\L_\eps\phi(x)
	=
	~&
	\frac{1}{2}\,\mathrm{Tr}\left\{D^2\phi(x)\cdot\left(\alpha\int z\otimes z\,d\nu_x(z)+\beta\vint_{B_1} y\otimes y\,dy\right)\right\}+o(\eps^0)
	\\
	=
	~&
	\frac{1}{2}\,\mathrm{Tr}\left\{D^2\phi(x)\cdot\left(\alpha\int z\otimes z\,d\nu_x(z)+\frac{\beta}{N+2}\,I\right)\right\}+o(\eps^0).
\end{split}
\end{equation*}
On the other hand, since every measure $\nu_x\in\M(B_\Lambda)$ defines a matrix
\begin{equation*}
	A(x)
	:\,=
	\frac{\alpha}{2}\int z\otimes z\,d\nu_x(z)+\frac{\beta}{2(N+2)}\, I,
\end{equation*}
we get that
\begin{equation*}
	\lim_{\eps\to 0}\L_\eps\phi(x)
	=
	\mathrm{Tr}\{D^2\phi(x)\cdot A(x)\},
\end{equation*}
which is a linear second order partial differential operator.
Furthermore, for $\beta\in(0,1]$, the operator is uniformly elliptic: given $\xi\in\R^N\setminus\{0\}$,
\begin{equation*}
	\langle A(x)\cdot\xi,\xi\rangle
	=
	\frac{\alpha}{2}\int |\langle z,\xi\rangle|^2\,d\nu_x(z)+\frac{\beta}{2(N+2)}\,|\xi|^2,
\end{equation*}
and estimating the integral we have that
\begin{equation*}
	\frac{\beta}{2(N+2)}
	\leq
	\frac{\langle A(x)\cdot\xi,\xi\rangle}{|\xi|^2}
	\leq
	\frac{\alpha\Lambda^2}{2}+\frac{\beta}{2(N+2)}.
\end{equation*}

It also holds, using Theorem~\ref{Holder} (cf.\ \cite[Theorem 4.9]{manfredipr12}), that under suitable regularity assumptions, the solutions $u:\,=u_{\eps}$  to the DPP converge  to a viscosity solution $v\in C(\Om)$ of 
\begin{align*}
\mathrm{Tr}\{D^2 v(x)\cdot A(x)\}=f(x),
\end{align*}
as $\eps\to 0$.
\end{example}

Similar convergence results also hold in the following examples.

\begin{example}[$p$-Laplacian]
\label{ex:plaplace}
In \cite[Section 3.2]{luirop17}, the following process that is covered by Theorem~\ref{Holder}  was considered. Let $u$ be a $p$-harmonic function whose gradient vanishes at most at a finite number of points. When at $x\in \Om$ such that $\nabla u(x)\neq 0$, we define a probability measure 
\begin{align*}
 \mu_{x,1}= \beta\mathcal{L}_{B_\eps(x)}+\frac{1-\beta}{2}\big(\delta_{x+\eps\frac{\nabla u(x)}{|\nabla u(x)|}}+
         \delta_{x-\eps\frac{\nabla u(x)}{|\nabla u(x)|}}\big)\,,
\end{align*}
where $\mathcal L_{B_\eps(x)}$ denotes the uniform probability distribution in $B_\eps(x)\subset\R^N$ and $\delta_x$ the Dirac measure at $x$. Then we choose the next point according to the probability measure 
\begin{equation*}
\mu_{x}=\begin{cases}
         \mu_{x,1},&\text{ if }|\nabla u(x)|> 0\,\,,\,\text{ and }\\
         \mathcal{L}_{B_\eps(x)},&\text{ if }|\nabla u(x)|= 0\,.
        \end{cases}
\end{equation*}

There is a classical well-known connection between the Brownian motion and the Laplace equation. This example is related to so called tug-of-war games introduced in the paper of Peres, Schramm, Sheffield and Wilson \cite{peresssw09} in connection with the infinity Laplace operator. Similarly, a connection exists between the $p$-Laplacian, $1<p<\infty$, and different variants of tug-of-war game with noise, \cite{peress08}, \cite{manfredipr12} and \cite{lewicka20}. 

There are several regularity methods devised for tug-of-war games with noise: the above papers contain global approach, and local approach is developed in \cite{luirops13} as well as \cite{luirop18}. However, none of these methods seem to directly apply to the present situation.
Moreover, later we prove Theorem~\ref{HolderL+L-} which applies to solutions of DPP associated to tug-of-war games with noise and the $p$-Laplacian, see Remark~\ref{aplications}.
\end{example}

\begin{example}[Ellipsoid process]
A particular case of the stochastic process considered in this paper is the so-called ellipsoid process (see \cite{arroyop}). This process arises when $\nu_x$ is the uniform probability measure on $E_x\setminus B_1$, where $E_x$ denotes an ellipsoid centered at the origin such that $B_1\subset E_x\subset B_\Lambda$. Then
\begin{equation*}
	\alpha\nu_x(A)+\beta\frac{|A\cap B_1|}{|B_1|}
	=
	\frac{|A\cap E_x|}{|E_x|}
\end{equation*}
for every measurable set $A\subset\R^N$, $\alpha=\frac{|E_x\setminus B_1|}{|E_x|}$ and $\beta=\frac{|B_1|}{|E_x|}$. Hence, replacing this in \eqref{DPP} with $f=0$ we get that the expectation related to the ellipsoid process satisfies the dynamic programming principle
\begin{equation*}
				u(x)
				=
				\vint_{E_x} u(x+\eps y)\,dy.
\end{equation*}
An asymptotic H\"older estimate was obtained in \cite{arroyop} under certain assumption on the ellipticity ratio of the ellipsoids. Now, Theorem~\ref{Holder} implies the H\"older estimate for $u$ without any additional assumption and thus improves the result in \cite{arroyop}.
Such mean value property over ellipsoids  has been studied by Pucci and Talenti in connection with smooth solutions to PDEs in \cite{puccit76}.
\end{example}

\subsection{Stochastic process}
\label{sec:stoc}
Next we define the stochastic process related to the DPP \eqref{DPP}. Let $x_0\in\Omega$ be the initial position of the process.
We equip $\R^N$ with the natural topology, and the $\sigma$-algebra $\mathcal B$ of the Borel measurable sets.
We consider along with the positions of the process the results of the coin tosses, so our process is defined in the product space
\[
H^\infty=\{x_0\}\times(\{0,1\}\times \R^N)\times(\{0,1\}\times \R^N)\times\cdots
\]
For $\omega= (x_0,(c_1,x_1),(c_2,x_2)\dots)\in H^\infty$, we define the result of the $k$-th toss $\mathcal C_k(\omega)=c_k$ and the coordinate processes $X_k(\omega)= x_k$.
If $\mathcal C_{k+1}=0$ (probability $\alpha$), the next position of the process $X_{k+1}$ is distributed according to $\nu_{x_k}$.
And for $\mathcal C_{k+1}=1$ (probability $\beta$), $X_{k+1}$ is uniformly distributed in $B_\eps(x_k)$.
That is, we have the following transition probabilities
\[
\pi(x_0,(c_1,x_1),\dots,(c_k,x_k),\{c\}\times A)=
\begin{cases}
\alpha \nu_{x_k}\left(\frac{A-x_k}{\eps}\right) &\text{if }c=0,\\
\beta \frac{|A\cap B_\eps(x_k)|}{|B_\eps|} &\text{if }c=1,
\end{cases}
\]
where $A\in\mathcal B$.

Let $\{\F_k\}_k$ denote the filtration of $\sigma$-algebras, $\F_0\subset \F_1\subset\cdots$ defined as follows:
$\F_k$ is the product $\sigma$-algebra generated by cylinder sets of the form 
\[
\{x_0\}\times A_1\times\cdots\times A_k\times\R^N\times\R^N\times\cdots
\]
with $A_i\in \mathcal P(\{0,1\})\times\mathcal B$.
We have that $\mathcal C_k$ and $X_k$ are $\F_k$-measurable random variables.
By Kolmogorov extension theorem, the transition probabilities determine a unique probability measure $\P^{x_0}$ in $H^\infty$ relative to the $\sigma$-algebra $
\mathcal F_\infty=\sigma(\cup_k\F_k)$.
We denote $\E^{x_0}$ the corresponding expectation.

We consider $\tau$ the exit time from the domain, that is $\tau=\min\{n\in\N:X_n\not\in \Omega\}$.
We define $T_A$ as the hitting time for $A$ and $\tau_A$ the exit time, that is
\[
T_A=\min \{k\in\N:X_k\in A\} \quad \text{and} \quad \tau_A=\min \{k\in\N:X_k\not\in A\}.
\]

\subsection{Stochastic estimates}
\label{time estimates}

In this section we establish some estimates related to $\tau$ and other stochastic results.
We will prove that $\E^{x_0}[\tau]$ is of order $1/\eps^2$.
Moreover, we will prove that the second moment of $\eps^2\tau$ is bounded.
We start with a rough estimate needed as a first step.

\begin{lemma}
The process leaves the domain almost surely and moreover
\[
\E^{x_0}[\tau]<+\infty
\]
for every $x_0\in\Omega$.
\end{lemma}

\begin{proof}
We fix $x_0\in\Omega$.
We will prove that there exists $\lambda<1$ and $n\in\N$ such that
\begin{equation}
\label{leaveinn}
\P^{x_0}(\tau> n+k|\tau > k)\leq\lambda
\end{equation}
for every $k\in\N$ and $x\in\Omega$.
Then, inductively  $\P^{x_0}(\tau> nk)\leq\lambda^k$ and we get
\[
\E^{x_0}[\tau]
=\sum_{i=0}^\infty \P^{x_0}(\tau> i)
=\sum_{k=0}^\infty \sum_{j=0}^{n-1} \P^{x_0}(\tau> j+nk)
\leq n \sum_{k=0}^\infty \P^{x_0}(\tau> nk)
\leq n \sum_{k=0}^\infty \lambda^k,
\]
which is finite.

Thus, it remains to prove \eqref{leaveinn}.
We choose $n$ such that $n\frac{\eps}{2}>\diam\Omega$.
We consider the event $E$ of $n$ steps after the $k$-th one to be uniformly distributed, and where the first coordinate increases at least $\frac{\eps}{2}$.
That is
\[
E=\{c_{k+1}=\cdots=c_{k+n}=1 \text{ and } \pi_1(x_{i}-x_{i-1})>\eps/2 \text{ for } i=k+1,\dots,k+n\},
\]
where $\pi_1$ denotes the projection to the first coordinate.
Observe that 
\[
\P^{x_0}(E)=\left(\beta \frac{|\{x:\pi_1(x)>\eps/2\}\cap B_\eps|}{|B_\eps|}\right)^n>0
\]
is independent of $k$.

Assuming $E$ we have that $|x_{k+n}-x_k|\geq n\frac{\eps}{2}$, hence since $n\frac{\eps}{2}>\diam\Omega$, it must be the case that $\tau\leq n+k$.
Therefore \eqref{leaveinn} holds for $\lambda=1-\P^{x_0}(E)<1$.
\end{proof}

Now we construct two sequences of random variables.
They will allow us to obtain bounds for the growth of the the expected value of the square of the distance from the starting point $x_0$.

\begin{lemma}
\label{super}
The sequence of random variables $\{|X_k-x_0|^2-Ck\eps^2\}_k$ is a supermartingale for $C=\alpha\Lambda^2+\beta\vint_{B_1}|x|^2\,dx$.
\end{lemma}

\begin{proof}
Observe that
\[
\E^{x_0}[|X_{k+1}-x_0|^2|\F_{k}](\omega)
=\alpha\int |x_{k}+\eps z-x_0|^2\,d\nu_{x_{k}}(z)+\beta \vint_{B_\eps(x_{k})} |x-x_0|^2\,dx.
\]
By the symmetry of $\nu_{x_k}$ and the ball $B_\eps$ we can write
\[
\begin{split}
\E^{x_0}[|X_{k+1}-x_0|^2|\F_{k}](\omega)
&=\alpha\int \frac{|x_{k}+\eps z-x_0|^2+|x_{k}-\eps z-x_0|^2}{2}\,d\nu_{x_{k}}(z)\\
&\quad\quad\quad+\beta \vint_{B_\eps} \frac{|x_k+x-x_0|^2+|x_k-x-x_0|^2}{2}\,dx.\\
\end{split}
\]
Employing the parallelogram law we get
\[
\begin{split}
\E^{x_0}[|X_{k+1}-x_0|^2|\F_{k}](\omega)
&=\alpha\int |x_{k}-x_0|^2+|\eps z|^2\,d\nu_{x_{k}}(z)\\
&\quad\quad\quad\quad\quad+\beta \vint_{B_\eps} |x_k-x_0|^2+|x|^2\,dx\\
&\leq \alpha(|x_{k}-x_0|^2+\eps^2\Lambda^2)+\beta\left(|x_{k}-x_0|^2+\eps^2\vint_{B_1}|x|^2\,dx\right)\\
&\leq |x_{k}-x_0|^2+\eps^2\left(\alpha\Lambda^2+\beta\vint_{B_1}|x|^2\,dx\right),\\
\end{split}
\]
where we have used that $\nu_{x_{k}}$ is supported in $B_{\Lambda}$.

Therefore, for $C=\alpha\Lambda^2+\beta\vint_{B_1}|x|^2\,dx$ we have
\begin{align*}
\E^{x_0}[|X_{k+1}-x_0|^2-C(k+1)\eps^2|\F_{k}]
&\leq |X_{k}-x_0|^2+C\eps^2-C(k+1)\eps^2\\
&= |X_{k}-x_0|^2-Ck\eps^2,
\end{align*}
as we wanted to show.
\end{proof}

\begin{lemma}
\label{sub}
The sequence of random variables $\{|X_k-x_0|^2-Ck\eps^2\}_k$ is a submartingale for $C=\beta\vint_{B_1}|x|^2\,dx$.
\end{lemma}

\begin{proof}
As in the previous lemma we have
\[
\begin{split}
\E^{x_0}[|X_{k+1}-x_0|^2|\F_{k}](\omega)
&=\alpha\int |x_{k}-x_0|^2+|\eps z|^2\,d\nu_{x_{k}}(z)\\
&\quad\quad\quad\quad\quad+\beta \vint_{B_\eps} |x_k-x_0|^2+|x|^2\,dx.\\
\end{split}
\]
By dropping the $|\eps z|^2$ term, we get
\[
\begin{split}
\E^{x_0}[|X_{k+1}-x_0|^2|\F_{k}](\omega)
&\geq \alpha|x_{k}-x_0|^2+\beta(|x_{k}-x_0|^2+\vint_{B_\eps}|x|^2\,dx)\\
&= |x_{k}-x_0|^2+\beta\eps^2\vint_{B_1}|x|^2\,dx.\\
\end{split}
\]
Therefore, for $C=\beta\vint_{B_1}|x|^2\,dx$ we have
\begin{align*}
\E^{x_0}[|X_{k+1}-x_0|^2-C(k+1)\eps^2|\F_{k}]
&\geq |X_{k}-x_0|^2+C\eps^2-C(k+1)\eps^2\\
&= |X_{k}-x_0|^2-Ck\eps^2,
\end{align*}
as claimed.
\end{proof}

We are ready to prove that $\E^{x_0}[\tau]$ is of order $1/\eps^2$.

\begin{lemma}
\label{taubounds}
There exists $C_1,C_2>0$ such that
\[
C_1 \dist(x_0,\partial\Omega)^2\leq \E^{x_0}[\eps^2\tau]\leq C_2 (\diam\Omega+1)^2
\]
for $\eps<\frac{1}{\Lambda}$.
\end{lemma}

\begin{proof}
We consider the supermartingale $M_k=|X_k-x_0|^2-Ck\eps^2$ from Lemma~\ref{super}.
Since the increments of $M_k$ are bounded and $\E^{x_0}[\tau]<+\infty$ we can apply the optional stopping theorem. 
We obtain $E^{x_0}[M_\tau]\leq 0$ and hence
\[
\dist(x_0,\partial\Omega)^2\leq E^{x_0}[|X_\tau-x_0|^2]\leq C\eps^2E^{x_0}[\tau]
\]
as desire.
The other inequality can be obtain by considering the submartingale from Lemma~\ref{sub}.
In fact, we get
\[
C\eps^2E^{x_0}[\tau]\leq  E^{x_0}[|X_\tau-x_0|^2] \leq  (\diam\Omega+\Lambda\eps)^2,
\]
where we have used that $x_\tau$ is at a distance of at most $\Lambda\eps$ from $x_{\tau-1}\in\Omega$ and therefore $|x_\tau-x_0| \leq  \diam\Omega+\Lambda\eps$.
\end{proof}

Now we obtain an estimate for the random variable $\eps^2\tau$ necessary to bound its second moment in the subsequent corollary.
We follow \cite[Lemma 3.6]{blancer}.
The key point here is that the process is memoryless.

\begin{lemma}
\label{expbound}
There exists $C>0$ and $\mu=\mu(\diam \Omega)>0$ such that
\[
\P^{x_0}(\eps^2\tau\geq t)\leq C e^{-\mu t}
\]
for $\eps<\frac{1}{\Lambda}$.
\end{lemma}

\begin{proof}
By Lemma~\ref{taubounds} there exists $\tilde C=\tilde C(\diam \Omega)>0$ such that $\E^{x_0}[\eps^2\tau]\leq \tilde C$.
Then, by Markov's inequality, we have
\[
\P^{x_0}(\eps^2\tau\geq t)\leq \frac{\E^{x_0}[\eps^2\tau]}{t}\leq \frac{\tilde C}{t}
\]
for $t>0$ and every $x_0\in\Omega$.
Observe that 
\[
\P^{x_0}(\eps^2\tau\geq \eps^2n+ t|\eps^2\tau\geq \eps^2n)
\leq
\sup_{x_n\in \Omega} \P^{x_n}(\eps^2\tau\geq t)
\leq \frac{\tilde C}{t}.
\]
So, for $n,k\in\N$, applying this bound multiple times we obtain
\begin{equation}
\label{eq:nk}
\begin{split}
\P^{x_0}(\eps^2\tau\geq \eps^2nk)
&=\P^{x_0}(\eps^2\tau\geq \eps^2nk|\eps^2\tau\geq \eps^2n(k-1))\\
&\quad\times \P^{x_0}(\eps^2\tau\geq \eps^2n(k-1)|\eps^2\tau\geq \eps^2n(k-2))\\
&\quad\times\cdots\times \P^{x_0}(\eps^2\tau\geq \eps^2n)\\
&\leq \left(\frac{\tilde C}{\eps^2n}\right)^k.
\end{split}
\end{equation}
We define $\delta=\tilde Ce+1$, and observe that
\[
\P^{x_0}(\eps^2\tau\geq t)\leq \P^{x_0}\left(\eps^2\tau\geq \eps^2\left\lfloor\frac{\delta}{\eps^2}\right\rfloor\left\lfloor\frac{t}{\delta}\right\rfloor\right).
\]
By \eqref{eq:nk} choosing $n=\left\lfloor\frac{\delta}{\eps^2}\right\rfloor$ and $k=\left\lfloor\frac{t}{\delta}\right\rfloor$, we get
\[
\P^x(\eps^2\tau\geq t)
\leq \left(\frac{\tilde C}{\eps^2\left\lfloor\frac{\delta}{\eps^2}\right\rfloor}\right)^{\left\lfloor\frac{t}{\delta}\right\rfloor}
\leq \left(\frac{\tilde C}{\delta-\eps^2}\right)^{\frac{t}{\delta}-1}
\leq e^{-\frac{t}{\delta}+1},
\]
where we have estimated $\eps^2\left\lfloor\frac{\delta}{\eps^2}\right\rfloor\geq \delta-\eps^2\geq \tilde Ce$.
The result holds for $C=e$ and $\mu=\frac{1}{\delta}$.
\end{proof}

\begin{corollary}
\label{corotau2}
There exists $C=C(\diam \Omega)>0$ such that
\[
\E^{x_0}[(\eps^2\tau)^2]\leq C
\]
for $\eps<\frac{1}{\Lambda}$.
\end{corollary}

\begin{proof}
By Lemma~\ref{expbound} there exist $\tilde C>0$ and $\mu=\mu(\diam \Omega)>0$ such that
\[
\P^{x_0}((\eps^2\tau)^2\geq t)
=
\P^{x_0}(\eps^2\tau\geq t^{1/2})
\leq \tilde C e^{-\mu t^{1/2}}.
\]
Then we can bound,
\[
\E^{x_0}[(\eps^2\tau)^2]
= \int_{0}^\infty \P^{x_0}((\eps^2\tau)^2\geq t)\ dt
\leq \int_{0}^\infty \tilde C e^{-\mu t^{1/2}}\ dt=\frac{2\tilde C}{\mu^2}.\qedhere
\]

\end{proof}

\section{Dynamic Programming Principle: Existence and uniqueness}
\label{sec:DPP}

Recall that $\Omega \subset \mathbb{R}^N$ is a bounded domain and $g :\mathbb{R}^N \setminus \Omega \to \mathbb{R}$ and $f :\Omega \to \mathbb{R}$ are measurable bounded functions.
We define
\[
u(x):\,=\mathbb{E}^{x} \left[\eps^2\sum_{i=0}^{\tau-1}f(X_i)+g (X_\tau)\right],
\]
where $\tau$ is the exit time from $\Omega$.
In this section we prove that $u$ is the unique bounded solution to the DPP \eqref{DPP} given by
\begin{align*}
u (x) =\alpha  \int u(x+\eps z) \,d\nu_x(z)+\beta\vint_{B_\eps(x)} u(y)\,dy+\eps^2 f(x)
\end{align*}
for $x \in \Omega$, and $u(x) = g(x)$ for $x \not\in \Omega$. For related arguments, see \cite{luirops14}, and  \cite{hartikainen16}, \cite{ruosteenoja16}, \cite{arroyohp17} as well as \cite{blancpr17}.

In the following lemma we prove that subsolutions are uniformly bounded.
We have required subsolutions to be bounded, which is necessary as shown by Example~\ref{ex:not-unique} below, but here we prove that there is a bound that only depends on the parameters of the problem and not the solution itself.

\begin{lemma}
\label{subsolbound}
There exist $C=C(\diam\Omega,f,g,\eps)>0$ such that $u\leq C$ for every subsolution $u$ to the DPP with boundary values $g$.
\end{lemma}

\begin{proof}
We consider the space partitioned along the $x_N$ axis in strips of width $\frac{\eps}{2}$.
We define $S_k=\{y:y_N<k\eps/2\}$,
\[
A=\frac{|\{y\in B_\eps(x):y_N<x_N-\frac{\eps}{2}\}|}{|B_\eps|}
=\frac{|\{y\in B_1:y_N<-\frac{1}{2}\}|}{|B_1|},
\]
and $K=\eps^2\sup_{\Omega} f$.
We have
\[
\begin{split}
u(x)
&\leq \alpha  \int u(x+\eps z) \,d\nu_x(z)+\beta\vint_{B_\eps(x)} u(y)\,dy+\eps^2 f(x)\\
&\leq \alpha  \sup u+\beta A \sup_{y:y_N< x_N-\frac{\eps}{2}} u(y) + \beta (1-A)\sup u+K\\
&=\beta A \sup_{y:y_N< x_N-\frac{\eps}{2}} u(y) + (1-\beta A)\sup u+K.
\end{split}
\]

We define $p=\beta A$ and consider $k\in \Z$, for $x\in S_{k+1}$
we have $\{y:y_N< x_N-\frac{\eps}{2}\}\subset S_k$.
Therefore we obtain
\[
\sup_{S_{k+1}} u
\leq p \sup_{S_k} u + (1-p)\sup u+K.
\]
Then, inductively, we get
\begin{equation}
\label{supineq}
\sup_{S_{n+k}} u
\leq \left((1-p) \sup u +K\right)\sum_{i=0}^{n-1}p^i+ p^n\sup_{S_k} u.
\end{equation}

We assume without loss of generality that $\Omega\subset\{y:0<y_N<R\}$ for some $R > 0$.
Then, since $u=g$ in $\R^N\setminus\Omega\supset S_0$, we have $\sup_{S_0} u=\sup_{S_0}g\leq \sup g$.
We assume that $\sup u\geq \sup g$ (if not then $\sup g$ is an upper bound for  $u$ and the proof would be finished) and consider $n$ such that $n\frac{\eps}{2}>R$, then $\Omega\subset S_n$ and we have $\sup_{S_n}u =\sup u$.
We apply \eqref{supineq} for such $n$ and $k=0$, we get
\[
\begin{split}
\sup u
&\leq \left((1-p) \sup u +K\right)\sum_{i=0}^{n-1}p^i+ p^n\sup g\\
&= \left((1-p) \sup u +K\right)\frac{1-p^n}{1-p}+ p^n\sup g\\
&= (1-p^n) \sup u +K\frac{1-p^n}{1-p}+ p^n\sup g.
\end{split}
\]
From where we finally get the upper bound
\[
\sup u
\leq K\frac{1-p^n}{p^n(1-p)}+\sup g.\qedhere
\]
\end{proof}

\begin{lemma}
\label{existencesub}
There exists $u_0$ a subsolution to the DPP with $u_0=g$ on $\R^N\setminus\Omega$.
\end{lemma}

\begin{proof}
We consider $v(x)=K+L|x|^2$ where we will choose $K\in\R$ and $L>0$ in what follows.
Since it is convex we have
\[
v(x)
\leq \int v(x+\eps z) \,d\nu_x(z).
\]
Even more, since it is strictly convex,
\[
v(x)
< \vint_{B_\eps(x)} v(y)\,dy,
\]
and therefore for $L$ large enough we get
\[
v(x)
\leq  \alpha\int v(x+\eps z) \,d\nu_x(z)+\beta \vint_{B_\eps(x)} v(y)\,dy -\eps^2\|f\|_\infty
\]
for every $x\in\Omega$.

We choose $K$ small enough such that $v\leq g$ in $\widetilde\Omega_{\Lambda\varepsilon}$.
Then $u_0$ given by $u_0=v$ on $\Omega$ and $u_0=g$ on $\R^N\setminus\Omega$ is a subsolution.
\end{proof}

\begin{proposition}
There exists a solution to the DPP with boundary values $g$.
\end{proposition}

\begin{proof}
We construct the solution by iterating the DPP.
We consider the function $u_0$ given by Lemma~\ref{existencesub} and define inductively 
\[
u_{n+1}(x)
= \alpha  \int u_n(x+\eps z) \,d\nu_x(z)+\beta\vint_{B_\eps(x)}  u_n(y)\,dy+\eps^2 f(x)
\]
for $x\in \Omega$ and $u_{n+1}=g$ on $\R^N\setminus\Omega$.

Since $u_0$ is a subsolution we get $u_1\geq u_0$.
Given $u_n \geq  u_{n-1}$, by the recursive definition, we get $u_{n+1} \geq  u_{n}$.
Then, by induction, we obtain that the sequence of functions is an increasing sequence.
By replacing $u_{n+1}$ by its definition in $u_{n+1} \geq  u_{n}$ we get that $u_n$ is a subsolution.

Then Lemma~\ref{subsolbound} gives us a uniform bound for $u_n$.
We conclude that the sequence of functions $u_n$ converges pointwise to a Borel function $u$.

Now our goal is to prove that $u$ is a solution to the DPP.
To that end, we will prove that the sequence converges uniformly.
Then we obtain that $u$ is a solution of the DPP by passing to the limit in the recursive definition.

For the sake of contradiction suppose that the convergence is not uniformly.
Observe that $\sup_\Omega u-u_n$ is a decreasing non negative sequence, then it has a limit
\[
M=\lim_{n\to\infty} \sup_\Omega u-u_n.
\]
Since we are assuming that the convergence is not uniform we have $M>0$.

Observe that by Fatou's Lemma, we have
\[
\lim_{n\to \infty}
\int_\Omega u(y)-u_n(y)\,dy=0
\]
so we can bound $\vint_{B_\eps(x)} u(y)-u_n(y)\,dy$ uniformly on $x$.

We fix $\delta>0$.
Let $n_1$ be such that 
\[
\sup_\Omega u-u_n\leq M+\delta
\]
for every $n\geq n_1$.
And let $n_2$ such that
\[
\vint_{B_\eps(x)} u(y)-u_n(y)\,dy\leq \delta 
\]
for every $x\in\Omega$ and  $n\geq n_2$.

Let $n_0=\max\{n_1,n_2\}$ and $k>n_0$, we have
\[
\begin{split}
u_k(x)&-u_{n_0+1}(x)\\
&=\alpha  \int u_{k-1}(x+\eps z)-u_{n_0}(x+\eps z) \,d\nu_x(z)
+\beta\vint_{B_\eps(x)}  u_{k-1}(y)-u_{n_0}(y)\,dy\\
&\leq \alpha  \int u(x+\eps z)-u_{n_0}(x+\eps z) \,d\nu_x(z)
+\beta\vint_{B_\eps(x)}  u(y)-u_{n_0}(y)\,dy\\
&=\alpha  (M+\delta) +\beta\delta
\end{split}
\]
for every $x\in\Omega$.
Since this holds for every $k>n_0$ we get 
\[
\sup_\Omega u-u_{n_0}<\alpha  (M+\delta) +\beta\delta.
\]
Recalling that $\alpha<1$ we can select $\delta$ such that $\alpha  (M+\delta) +\beta\delta<M$ and we have reached a contradiction.
\end{proof}

\begin{theorem}
\label{DPPsol}
The function 
\[
u(x)=\mathbb{E}^{x} \left[\eps^2\sum_{i=0}^{\tau-1}f(X_i)+g (X_\tau)\right]
\]
is the unique bounded solution to the DPP with boundary values $g$.
\end{theorem}

\begin{proof}
Given a solution $v$ to the DPP, we have that $\{v(X_k)+\eps^2\sum_{i=0}^{k-1}f(X_i)\}_k$ is a martingale. Indeed,
\[
\begin{split}
\E^{x_0}[v(x_{k+1})+&\eps^2\sum_{i=0}^{k}f(x_i)|\F_{k}](\omega)\\
&=\alpha  \int v(x_k+\eps z) \,d\nu_{x_k}(z)+\beta\vint_{B_\eps(x_k)} v(y)\,dy+\eps^2\sum_{i=0}^{k}f(x_i)\\
&=v(x_k)+\eps^2\sum_{i=0}^{k-1}f(x_i).
\end{split}
\]
Then, by Doob's stopping time theorem (recall that $v$ and $f$ are bounded) we have
\[
v(x)=\mathbb{E}^x\left[v(x_\tau)+\eps^2\sum_{i=0}^{\tau-1}f(x_i)\right].
\]
Hence, since $v(x) = g(x)$ for $x \not\in \Omega$ we obtain $v(x)=\mathbb{E}^x[g(x_\tau)+\eps^2\sum_{i=0}^{\tau-1}f(x_i)]$.
Thus $v=u$, we have proved that $u$ is a solution to the DPP and that every solution coincides with it, there is a unique solution. 
\end{proof}

\begin{example}
\label{ex:not-unique}
The uniqueness fails if we do not assume that solutions to the DPP are bounded.
For $\Omega=(-2,2)\subset \R$ we consider the process given by $\eps=1$, $\alpha=\beta=1/2$ and $\nu_x$ the uniform probability distribution on $B_1$ for every point $x\in \Om$ except the ones in the set $\{\frac{1}{2^k}\}_{k\in\N}$. There we set
$$
 \nu_\frac{1}{2^k}=\frac{\delta_\frac{1}{2^{k+1}}+\delta_\frac{3}{2^{k+1}}}{2}.
$$
In this case the DPP has multiple solutions: the function $u\equiv 0$ and
\begin{align*}
v(x)=\begin{cases}
 4^k& x=\frac{1}{2^k}, \ k\in \mathbb N,\\
 0& \text{otherwise}, 
\end{cases}
\end{align*}
which is not bounded.
This is why the solutions to the DPP are required to be bounded. 
\end{example}

\section{$\eps$-ABP estimate}
\label{sec-ABP}

Regarding the classic theory of elliptic PDEs, one of the key inequalities in the Krylov-Safonov proof of H\"older regularity is the so-called Aleksandrov-Bakelman-Pucci estimate (ABP estimate for short), which guarantees a pointwise bound for subsolutions of $Lu+f=0$ by means of the $L^N$-norm of $f$. Namely, if $f$ is a continuous bounded function in $\Omega$ and  $u\in C(\overline\Omega)$ satisfies
\begin{equation*}
	\mathrm{Tr}\{A(x)\cdot D^2 u(x)\} + f(x) \geq 0, \qquad x\in\Omega,
\end{equation*}
then there exists a constant $C>0$ depending only on $N$, $\diam\Omega$ and the uniform ellipticity of $A(x)$ such that
\begin{equation*}
	\sup_\Omega u
	\leq
	\sup_{\partial\Omega}u+C\,\|f\|_{L^N(\Omega)}.
\end{equation*}
See \cite[Chapter 3]{caffarellic95} for the ABP estimate for viscosity solutions and \cite[Section 9.1]{gilbargt01} for strong solutions. 

Given a subsolution $u$, one of the key ideas in the proof of the classical ABP estimate was the use of the concavity properties of $u$ at the set of points where the graph of $u$ can be \emph{touched} from above by tangent hyperplanes. This set of points (known as the \emph{contact set} and denoted by $K_u$) turned out to carry all relevant information about the subsolution. To be more precise, if we denote by $\Gamma$ the concave envelope of $u$, the ABP estimate is obtained by studying the behavior of $\Gamma$ at those points in $\Omega$ where $\Gamma$ and $u$ agree. Using the concavity of $\Gamma$, the first main step in the proof consisted on obtaining an estimate of $\sup_\Omega u$ in terms of $|\nabla\Gamma(\Omega)|$. It is worth noting that the structure of the PDE does not play any role in the proof of this first estimate, which was obtained using exclusively geometric arguments. 

In a second step, and in addition to the concavity of $\Gamma$, it turned out that $\Gamma$ is $C^{1,1}$ in the contact set so, by virtue of Rademacher's theorem, $\Gamma$ is indeed $C^2$ a.e. in $K_u$. This fact and a change of variables formula gives an inequality of the form
\begin{equation*}
	|\nabla\Gamma(\Omega)|
	\leq
	\int_\Omega|\det D^2\Gamma(x)|\,dx,
\end{equation*}
which allowed to use the equation to estimate the right hand side and, consequently, to obtain the ABP estimate.

However, in the case of the DPP, the non-local nature of the setting forces us to consider also non-continuous subsolutions of the DPP, so the corresponding concave envelope $\Gamma$ might not be $C^{1,1}$ as in the classical setting. In addition there is no PDE to connect with the right hand side of the previous inequality, and therefore we follow a different strategy in order to estimate $|\nabla\Gamma(\Omega)|$. The idea is to cover the contact set $K_u$ by a finite collection of balls of radius $\eps/4$, and then to estimate $|\nabla\Gamma(B_{\eps/4}(x))|$ by means of the oscillation of $\Gamma$ with respect to a supporting hyperplane touching the graph of $\Gamma$ from above at $x\in K_u$. This oscillation, in turn, is estimated by using the DPP, which yields the desired $\eps$-ABP estimate. It is also interesting to note that one can recover the classical ABP estimate by taking limits as $\eps\to 0$.

In this section we adapt the ideas from \cite[Section 8]{caffarellis09}, where an ABP-type estimate was obtained for continuous solutions of non-local integro-differential equations (see also \cite{caffarellilu14} and \cite{caffarellitu20} for similar approaches). Further references are \cite{kuot90} for an ABP estimate for elliptic difference equations and \cite{guillens12}, where an ABP-type estimate is obtained using a generalization of the concept of concave envelope as a non-local fractional envelope. 

Let $u:\R^N\to\R$ be a bounded function and let $\Gamma$ be the \emph{concave envelope} of $u^+:\,=\max\{u,\sup_{\R^N\setminus\Omega}u\}$ in $\widetilde\Omega_{\Lambda\varepsilon}$, that is
\begin{equation*}
		\Gamma(x)
		:\,=
		\left\{
		\begin{array}{ll}
			\inf\set{\ell(x)}{\mbox{ for all hyperplanes } \ell\geq u^+ \mbox{ in } \widetilde\Omega_{\Lambda\varepsilon}}
			&
			\mbox{ if } x\in\widetilde\Omega_{\Lambda\eps},
			\\
			\sup_{\R^N\setminus\Omega}u
			&
			\mbox{ if } x\notin\widetilde\Omega_{\Lambda\eps}.
		\end{array}
		\right.
\end{equation*}
Since $u$ is not necessarily continuous, we define the `contact' set  as
\begin{equation}\label{contact-set}
		K_u
		:\,=
		\set{x\in\overline\Omega}{\limsup_{y\to x}u^+(y)=\Gamma(x)}.
\end{equation}
Since $u^+\leq\Gamma$, then $K_u$ is a closed subset, and thus compact.
Moreover, observe that in the particular case of $u$ being an upper semicontinuous function in $\Omega$, then $K_u=\overline\Omega\cap\{u^+=\Gamma\}$.

As we have already pointed out, one of the key steps in the proof of our ABP-type estimate is the construction of a suitable cover of the contact set $K_u$ by balls of radius $\eps/4$. For this purpose, before stating the main result of this section we introduce the following notation. Given $\eps>0$, we denote by $\mathcal{Q}_\eps(\R^N)$ a grid of open cubes of diameter $\eps/4$ covering $\R^N$. Take for instance
\begin{equation*}
	\mathcal{Q}_\eps(\R^N)
	:\,=
	\{Q=Q_{\frac{\eps}{4\sqrt{N}}}(x)\,:\,x\in\frac{\eps}{4\sqrt{N}}\,\Z^N\}.
\end{equation*}
In addition, if $A\subset\R^N$, we write
\begin{equation}\label{Q_eps}
	\mathcal{Q}_\eps(A)
	:\,=
	\{Q\in\mathcal{Q}_\eps(\R^N)\,:\,\overline Q\cap A\neq\emptyset\},
\end{equation}
so
\begin{equation*}
	A
	\subset
	\bigcup_{Q\in\mathcal{Q}_\eps(A)}\overline Q.
\end{equation*}

We stress that, while not needed in the proof of the main $\eps$-ABP estimate, the assumption of $\mathcal{Q}_\eps$ being a grid is needed later in the proof of Theorem~\ref{boundomega}.

Now we are in conditions of stating the main theorem of this section.  We use the notation $\L_\eps u + f$ for convenience in some of the proofs, but this is equivalent to the DPP and stochastic notations as we recall at the end of the section.

\begin{theorem}[$\eps$-ABP estimate with continuous $f$]\label{ABP estimate}
Suppose that $u$ is a bounded Borel measurable function satisfying
\begin{equation*}
		\L_\eps u + f
		\geq
		0
\end{equation*}
in $\Omega$ for $f\in C(\overline\Omega)$. Let $\Gamma$ be the concave envelope of $u^+$ in $\widetilde\Omega_{\Lambda\eps}$ and let $\mathcal{Q}=\mathcal{Q}_\eps(K_u)$ be the grid of pairwise disjoint open cubes $Q$ of diameter $\eps/4$ defined in \eqref{Q_eps}.
Then
\begin{equation}\label{ABP}
		\sup_\Omega u
		\leq
		\sup_{\R^N\setminus\Omega} u + \frac{2^{N+3}}{\beta}(\diam\Omega+\Lambda\eps)\bigg(\sum_{Q\in\mathcal{Q}}(\sup_Qf^+)^N\bigg)^{1/N}\eps.
\end{equation}
\end{theorem}

After proving this theorem we relate the result to the stochastic process and in Theorem~\ref{boundomega} we obtain a version of the estimate where the continuity hypothesis for $f$ is removed.

In what follows, we can assume without loss of generality that $f\geq 0$ in $\overline\Omega$ and $\sup_{\R^N\setminus\Omega}u=0$. Then $u^+=\max\{u,0\}$.

It turns out that in order to prove Theorem~\ref{ABP estimate} we only need to use the information of the concave envelope in the set of contact points $K_u$. Indeed, since
\begin{equation*}
	\delta u(x,y)
	\leq
	\delta\Gamma(x,y)+2[\Gamma(x)-u(x)],
\end{equation*}
inserting this inequality in \eqref{L-eps} we get that
\begin{equation*}
		\L_\eps u(x)
		\leq
		\L_\eps\Gamma(x)
		+\frac{1}{\eps^2}[\Gamma(x)-u(x)].
\end{equation*}
Since $\Gamma$ is concave, then $\Gamma$ is continuous and for any fixed $x_0\in K_u$ we have that
\begin{equation*}
		\liminf_{x\to x_0}
		\L_\eps u(x)
		\leq
		\L_\eps\Gamma(x_0).
\end{equation*}
Hence, if $f\geq 0$ is a continuous function in $\overline\Omega$ and $u$ is a bounded function satisfying
\begin{equation}\label{Lu+f>0}
		\left\{
		\begin{array}{rl} 
			\L_\eps u + f \geq 0 & \mbox{ in $\Omega$,}
			\\
			u\leq 0 & \mbox{ in $\R^N\setminus\Omega$,}
		\end{array}
		\right.
\end{equation}
then
\begin{equation}\label{LGamma+f>0}
		\L_\eps\Gamma+f\geq 0 \quad \mbox{ in $K_u$}.
\end{equation}
Therefore, in what follows, we will use \eqref{LGamma+f>0} instead of \eqref{Lu+f>0}.

Before stating the first lemma of this section, we need to define the \emph{superdifferential}  of $\Gamma$ at $x\in\widetilde\Omega_{\Lambda\eps}$ as the set
\begin{equation}\label{super-diff}
		\nabla\Gamma(x)
		:\,=
		\set{ \xi\in\R^N }{\Gamma(z)\leq\Gamma(x)+\prodin{\xi}{z-x} \ \mbox{ for all } z\in\widetilde\Omega_{\Lambda\eps}}.
\end{equation}
Since $\Gamma$ is a concave function in $\widetilde\Omega_{\Lambda\eps}$, then $\nabla\Gamma(x)\neq\emptyset$ for every $x\in\widetilde\Omega_{\Lambda\eps}$. Moreover, given a set $S\subset\widetilde\Omega_{\Lambda\eps}$, we denote $\nabla\Gamma(S)=\bigcup_{x\in S}\nabla\Gamma(x)$.

In addition, if $S$ is a compact subset of $\overline\Omega$, then $\nabla\Gamma(S)$ is closed. Indeed, if $\{\xi_n\}_n\subset\nabla\Gamma(S)$ is a sequence converging to $\xi_0\in\widetilde\Omega_{\Lambda\eps}$, by definition there exists $\{x_n\}_n\subset S$ (which by compactness we can assume that converges to some $x_0\in S$ by passing to a subsequence) such that $\Gamma(z)\leq\Gamma(x_n)+\prodin{\xi_n}{z-x_n}$ for each $z\in\widetilde\Omega_{\Lambda\eps}$. Since $\Gamma$ is concave (and thus continuous), taking limits we get that $\xi_0\in\nabla\Gamma(x_0)\subset\nabla\Gamma(S)$. In consequence, $\nabla\Gamma(S)$ is a Lebesgue measurable set.

\begin{lemma}\label{sup u - nabla Gamma}
Let $u:\R^N\to\R$ be a bounded function such that $u\leq 0$ in $\R^N\setminus\Omega$. Then
\begin{equation}\label{key1}
		\sup_{\Omega}u
		\leq
		(\diam\Omega+\Lambda\varepsilon)\,\bigg(\frac{|\nabla\Gamma(K_u)|}{|B_1|}\bigg)^{1/N},
\end{equation}
where $K_u$ is the contact set defined in \eqref{contact-set}.
\end{lemma}

\begin{proof}
Let us assume that $\sup_\Omega u>0$ (otherwise \eqref{key1} would follow trivially) and define
\begin{equation*}
		\rho
		:\,=
		\frac{\sup_\Omega u}{\diam\Omega+\Lambda\eps}
		>
		0.
\end{equation*}
Since $\Gamma$ is the concave envelope of $u^+$ and $\sup u=\sup\Gamma$, then for every $|\xi|<\rho$ there exists $\ell$ a supporting hyperplane of $\Gamma$ in $\widetilde\Omega_{\Lambda\eps}$ such that $\nabla\ell\equiv\xi$. 
Fix any $\xi\in B_\rho$. We claim that $\xi\in\nabla\Gamma(x_0)$ for some $x_0\in K_u$. To see this, let
\begin{equation*}
		\tau
		:\,=
		\sup_{z\in\widetilde\Omega_{\Lambda\eps}}\{u^+(z)-\prodin{\xi}{z}\}
\end{equation*}
and define $\ell(z)=\tau+\prodin{\xi}{z}$ for every $z\in\widetilde\Omega_{\Lambda\eps}$. Then $\ell\geq\Gamma\geq u^+$ in $\widetilde\Omega_{\Lambda\eps}$. Moreover, by the definition of $\tau$, for each $n\in\N$, there exists $x_n\in\widetilde\Omega_{\Lambda\eps}$ such that
\begin{equation*}
		u^+(x_n)+\frac{1}{n}
		\geq
		\ell(x_n)
		\geq
		\Gamma(x_n).
\end{equation*}
On the other hand, by the definition of $\ell$, for any $z\in\Omega$ we have
\begin{equation*}
\begin{split}
		u^+(z)
		\leq
		\ell(z)
		=
		~&
		\ell(x_n)+\prodin{\xi}{z-x_n}
		\\
		\leq
		~&
		\ell(x_n)+|\xi|(\diam\Omega+\Lambda\eps)
		\\
		=
		~&
		\ell(x_n)+\frac{|\xi|}{\rho}\sup_\Omega u^+
\end{split}
\end{equation*}
for each $n\in\N$, where in the inequality we have used that $x_n\in\widetilde\Omega_{\Lambda\eps}$ and $z\in\Omega$, and the definition of $\rho$ has been recalled in the last equality. Taking the supremum for $z\in\Omega$ we obtain
\begin{equation*}
		\bigg(1-\frac{|\xi|}{\rho}\bigg)\sup_\Omega u^+
		\leq
		\ell(x_n)
		\leq
		u^+(x_n)+\frac{1}{n}
\end{equation*}
for each $n\in\N$.
Hence, since $|\xi|<\rho$ and $\sup_\Omega u^+>0$, we can assume that $x_n\in\overline\Omega$ for every $n\in\N$. Otherwise, since $u\leq 0$ in $\R^N\setminus\Omega$ by assumption, if $x_n\in\widetilde\Omega_{\Lambda\eps}\setminus\overline\Omega$ for each sufficiently large $n\in\N$, letting $n\to\infty$ we would obtain a contradiction. Furthermore, by a compactness argument, we can assume without loss of generality that $x_n$ converges to a point $x_0\in\overline\Omega$. Thus, since $\Gamma$ is continuous, taking limits we get
\begin{equation*}
		\limsup_{y\to x_0}u^+(y)
		\geq
		\limsup_{n\to\infty}u^+(x_n)
		\geq
		\Gamma(x_0).
\end{equation*}
Finally, since $u^+\leq\Gamma$, we have in particular that $x_0\in K_u$ and $\xi\in\nabla\Gamma(x_0)$. In consequence $B_\rho\subset\nabla\Gamma(K_u)$, so $|B_1|\rho^N\leq|\nabla\Gamma(K_u)|$ and \eqref{key1} follows.
\end{proof}

The idea is to estimate the term $|\nabla\Gamma(K_u)|$ in the right hand side of \eqref{key1} by covering the contact set $K_u$ with balls of radius $\eps/4$ and estimating $|\nabla\Gamma(B_{\eps/4}(x))|$. This is done by obtaining an upper bound for the gradients of the concave function $\Gamma$ in $B_{\eps/4}(x)$ which depends on the oscillation of $\Gamma$ with respect to a supporting hyperplane touching the graph of $\Gamma$ at $x$.

\begin{lemma}\label{nabla Gamma - osc}
Let $\Gamma:\widetilde\Omega_{\Lambda\eps}\to\R$ be a concave function. Then
\begin{equation}\label{key3}
		\frac{|\nabla\Gamma(B_{\eps/4}(x))|}{|B_\eps|}
		\leq
		\Big(\frac{2}{\eps^2}\,\osc_{y\in B_{\eps/2}}\big\{\Gamma(x)-\Gamma(x+y)+\prodin{\xi}{y}\big\}\Big)^N
\end{equation}
for every $x\in\Omega$ and $\xi\in\nabla\Gamma(x)$.
\end{lemma}

\begin{proof}
Fix $x\in\Omega$ and any $\xi\in\nabla\Gamma(x)$ and define the auxiliary function $\Phi:B_{\eps/2}\to\R$ by
\begin{equation*}
		\Phi(y)
		:\,=
		\Gamma(x+y)-\Gamma(x)-\prodin{\xi}{y}
\end{equation*}
for every $y\in B_{\eps/2}$. Since $\Gamma$ is a concave function, then $\Phi$ is also concave in $B_{\eps/2}$, so
\begin{equation*}
		\nabla\Phi(y)
		=
		\set{ \zeta\in\R^N }{\Phi(y+z)\leq\Phi(y)+\prodin{\zeta}{z} \ \mbox{ for all }\ z\ \mbox{ s.t. }\ y+z\in B_{\eps/2}}
		\neq
		\emptyset
\end{equation*}
for every $y\in B_{\eps/2}$. Let us fix any $y\in B_{\eps/4}$ and any $\zeta\in\nabla\Phi(y)$. Since $\Phi$ is concave, $\Phi\leq 0$ and $\Phi(0)=0$ we have
\begin{equation}\label{osc Phi}
		|\zeta|
		\leq
		\frac{\osc_{B_{\eps/2}\setminus B_{\eps/4}}\Phi}{\eps/2}
		\leq
		\frac{2}{\eps}\,\osc_{B_{\eps/2}}\Phi
		=\,:
		\rho.
\end{equation}
On the other hand, by the definition of $\Phi$ and $\nabla\Gamma(x+y)$ in \eqref{super-diff}, for any $\xi'\in\nabla\Gamma(x+y)$ we have
\begin{equation*}
		\Phi(y+z)-\Phi(y)-\prodin{\xi'-\xi}{z}
		=
		\Gamma(x+y+z)-\Gamma(x+y)-\prodin{\xi'}{z}
		\leq
		0,
\end{equation*}
for every $z$ such that $z+y\in B_{\eps/2}$,
so $\xi'-\xi\in\nabla\Phi(y)$. Then, since $|\zeta|\leq\rho$ for every $\zeta\in\nabla\Phi(y)$ by \eqref{osc Phi}, we get that $\xi'\in\overline B_\rho(\xi)$ for every $\xi'\in\nabla\Gamma(x+y)$. Thus $\nabla\Gamma(x+y)\subset\overline B_\rho(\xi)$ for every $y\in B_{\eps/4}$, so $|\nabla\Gamma(B_{\eps/4}(x))|\leq|B_1|\rho^N$ and \eqref{key3} follows.
\end{proof}

The following lemma shows that the graph of $\Gamma$ stays quadratically close to a tangent hyperplane in a neighborhood of any point in which the inequality $\L_\eps\Gamma+f\geq 0$ is satisfied. It is noteworthy to mention that this is the only result where the DPP is used.

\begin{lemma}\label{key lemma}
Suppose that $\Gamma$ is a concave function and $x_0\in\overline\Omega$ satisfying $\L_\eps\Gamma(x_0)+f(x_0)\geq 0$. Then, for any $w>0$, the following holds
\begin{equation}\label{Gamma level set estimate}
		\frac{|\set{y\in B_\eps}{\Gamma(x_0)-\Gamma(x_0+y)+\prodin{\xi}{y}> w\, }|}{|B_\eps|}
		\leq
		\frac{f(x_0)\,\varepsilon^2}{w\beta},
\end{equation}
where $\xi$ is any vector in $\nabla\Gamma(x_0)$.
Furthermore,
\begin{equation}\label{osc-est}
		\osc_{y\in B_{\eps/2}}\{\Gamma(x_0)-\Gamma(x_0+y)+\prodin{\xi}{y}\}
		\leq
		\frac{2^{N+2}}{\beta}f(x_0)\,\varepsilon^2.
\end{equation}
\end{lemma}

\begin{proof}
First observe that, since $\Gamma$ is concave in $\widetilde\Omega_{\Lambda\eps}$, then $\delta\Gamma(x_0,y)\leq 0$ for every $y\in B_{\Lambda\eps}$. Thus, we can estimate by zero the $\alpha$-term in \eqref{L-eps}, so we obtain
\begin{equation}
\label{Gamma-ineq}
		\L_\eps\Gamma(x_0)
		\leq
		\frac{\beta}{2\eps^2}\vint_{B_\eps}\delta\Gamma(x_0,y)\,dy.
\end{equation}
Since $f(x_0)\geq-\L_\eps\Gamma(x_0)$ by assumption, using the definition of $\delta\Gamma(x_0,y)$ and the symmetry of the ball
we can estimate
\begin{equation*}
\begin{split}
		\frac{f(x_0)\,\eps^2}{\beta}
		&
		\geq
		-\frac{1}{2}\vint_{B_\eps}\delta\Gamma(x_0,y)\,dy
		\\
		&
		=
		\vint_{B_\eps}(\Gamma(x_0)-\Gamma(x_0+y))\, dy
		\\
		&
		=
		\vint_{B_\eps}(\Gamma(x_0)-\Gamma(x_0+y)+\prodin{\xi}{y})\, dy,
\end{split}
\end{equation*}
for any fixed $\xi\in\nabla\Gamma(x_0)$.
Let us define the auxiliary function $\Phi:B_\eps\to\R$ by 
\begin{equation*}
		\Phi(y)
		=
		\Gamma(x_0)-\Gamma(x_0+y)+\prodin{\xi}{y}.
\end{equation*}
Observe that, for the sake of convenience, the sign of $\Phi$ has been changed with respect to the previous proof. Notice that $\Phi\geq 0$ due to the concavity of $\Gamma$. We split the ball $B_\eps$ in two sets and we study the integral of $\Phi$ over each of them. Then
\begin{equation*}
\begin{split}
		\int_{B_\eps}\Phi(y)\, dy
		=
		~&
		\int_{B_\eps\cap\{\Phi>w\}}\Phi(y)\, dy+\int_{B_\eps\cap\{\Phi\leq w\}}\Phi(y)\, dy
		\\
		\geq
		~&
		\int_{B_\eps\cap\{\Phi>w\}}w\, dy
		\\
		=
		~&
		w\,\big|B_\eps\cap\{\Phi>w\}\big|,
\end{split}
\end{equation*}
where in the inequality we used that $\Phi\geq 0$ to estimate the second integral over $B_\eps\cap\{\Phi\leq w\}$. Then \eqref{Gamma level set estimate} follows by combination of the previous estimates.

Now we prove \eqref{osc-est}.
If $f(x_0)=0$, then \eqref{Gamma level set estimate} yields that $\Phi\leq w$ a.e. for every $w>0$. Then, since $\Phi$ is continuous and $\Phi\geq 0$, we get that $\Phi\equiv 0$, so the oscillation in \eqref{osc-est} is zero as desired.

If $f(x_0)>0$, we choose $w>0$ so that $0\leq\Phi(y)\leq w$ holds for every $y\in B_{\eps/2}$. Notice that, as we already mentioned, $\Phi\geq 0$ follows directly from the concavity of $\Gamma$. To check that $\Phi\leq w$, observe that the inclusion $B_{\eps/2}(y)\subset B_\eps$ holds for every $y\in B_{\eps/2}$. Then \eqref{Gamma level set estimate} yields
\begin{equation*}
		\frac{|\set{z\in B_{\eps/2}}{\Phi(y+z)> w\, }|}{|B_{\eps/2}|}
		\leq
		2^N\frac{|B_\eps\cap\{\Phi> w\}|}{|B_\eps|}
		\leq
		2^N\frac{f(x_0)\,\varepsilon^2}{w\beta}.
\end{equation*}
In particular, choosing $w=\frac{2^{N+2}}{\beta}f(x_0)\,\varepsilon^2$, we get that the left hand side of the previous inequality is bounded by $1/4$, and thus there exists $z\in B_{\eps/2}$ such that
\begin{equation*}
		\Phi(y\pm z)
		\leq
		w
		=
		\frac{2^{N+2}}{\beta}f(x_0)\,\varepsilon^2.
\end{equation*}

Combining the inequalities for $z$ and $-z$ we obtain that
\begin{equation*}
		\frac{1}{2}\Phi(y+z)+\frac{1}{2}\Phi(y-z)
		\leq
		w,
\end{equation*}
so $\Phi(y)\leq w$ follows from the convexity of $\Phi$ and this completes the proof.
\end{proof}

Now we are in conditions of proving the main result of this section.

\begin{proof}[Proof of Theorem~\ref{ABP estimate}]
Let us consider the pairwise disjoint collection of open cubes $\mathcal{Q}_\eps(K_u)$ defined in \eqref{Q_eps}. Then the following conditions are satisfied:
\begin{enumerate}
		\item $\diam Q=\eps/4$;
		\item $\overline{Q}\cap K_u\neq\emptyset$ for each $Q\in\mathcal{Q}_\eps(K_u)$;
		\item $\displaystyle\Omega\subset\bigcup_{Q\in\mathcal{Q}_\eps(K_u)}\overline{Q}$.
\end{enumerate}
Since $K_u$ is bounded, we can label de cubes in $\mathcal{Q}_\eps(K_u)$ as $Q_1,\ldots,Q_n$, where $n=n(\eps)\in\N$.
Furthermore, we select a point $x_i\in K_u\cap Q_i$ for each $i=1,\ldots,n$ so that $Q_i\subset B_{\eps/4}(x_i)$. From all the above considerations we can estimate
\begin{equation*}
		|\nabla\Gamma(K_u)|
		\leq
		\bigg|\bigcup_{i=1}^n\nabla\Gamma(\overline{Q}_i)\bigg|
		\leq
		\sum_{i=1}^n|\nabla\Gamma(\overline Q_i)|
		\leq
		\sum_{i=1}^n|\nabla\Gamma(B_{\eps/4}(x_i))|.
\end{equation*}
Combining this with the estimates from Lemmas~\ref{nabla Gamma - osc} and~\ref{key lemma} we obtain
\begin{equation*}
		|\nabla\Gamma(K_u)|
		\leq
		|B_1|\pare{\frac{2^{N+3}}{\beta}}^N\bigg(\sum_{i=1}^nf(x_i)^N\bigg)\eps^N.
\end{equation*}
Moreover, since $x_i\in Q_i$, we can estimate $f(x_i)\leq\sup_{Q_i}f$ for each $i=1,\ldots,n$. Then the result follows by replacing this in the estimate from Lemma~\ref{sup u - nabla Gamma}.
\end{proof}

As we have seen in Section~\ref{sec:DPP}, solutions to the DPP can be interpreted as an expected value. Thus the $\eps$-ABP extends to this setting as well.

\begin{corollary}
\label{boundomegafcont}
Given $f\in C(\overline\Omega)$ such that $f\geq 0$ and
the family $\mathcal Q=\mathcal Q_\eps(\Omega)$ of pairwise disjoint open cubes $Q$ of diameter $\eps/4$ defined in \eqref{Q_eps},
there exist $C>0$ and such that
\begin{equation}
\label{festimate}
\E^x\left[\eps^2\sum_{i=0}^{\tau-1}f(X_i)\right]
\leq C(\diam\Omega+\Lambda\eps)\bigg(\sum_{Q\in\mathcal Q}(\sup_Q f)^N\bigg)^{1/N}\eps.
\end{equation}
More precisely, $C=2^{N+3}/\beta$.
\end{corollary}

\begin{proof}
We consider $u:\R^N\to \R$
\[
u(x)=\E^x\left[\eps^2\sum_{i=0}^{\tau-1}f(X_i)\right].
\]
By Theorem~\ref{DPPsol} we have that $\L_\eps u=-f$ in $\Omega$ and $u=0$ in $\R^N\setminus\Omega$.
The result follows by applying Theorem~\ref{ABP estimate} to $u$.
\end{proof}

Let us observe that the $\eps$-ABP estimate \eqref{ABP} yields the classical ABP estimate after taking limits as $\eps\to 0$. Since each $Q$ in $\mathcal{Q}_\eps(K_u)$ has diameter $\eps/4$, then $\eps=4\sqrt{N}|Q|^{1/N}$ and thus
\begin{equation*}
		\bigg(\sum_{Q\in\mathcal{Q}_\eps(K_u)}(\sup_Qf^+)^N\bigg)^{1/N}\eps
		=
		4\sqrt{N}\bigg(\sum_{Q\in\mathcal{Q}_\eps(K_u)}(\sup_Qf^+)^N|Q|\bigg)^{1/N}.
\end{equation*}
Furthermore, letting $\eps\to 0$ the size of the cubes in $\mathcal{Q}_\eps(K_u)$ converges to zero, and since $f$ is continuous, we obtain the $L^N(K_u)$-norm of $f^+$ as the limit of Riemann sums, i.e.
\begin{equation*}
		\lim_{\eps\to 0}\bigg(\sum_{Q\in\mathcal{Q}_\eps(K_u)}(\sup_Qf^+)^N\bigg)^{1/N}\eps
		=
		4\sqrt{N}\,\|f^+\|_{L^N(K_u)},
\end{equation*}
where
\begin{equation*}
		\|f^+\|_{L^N(K_u)}
		:\,=
		\bigg(\int_{K_u}f^+(x)^N\,dx\bigg)^{1/N}.
\end{equation*}
Thus, replacing this in the $\eps$-ABP estimate \eqref{ABP} we get
\begin{equation*}
		\sup_\Omega u
		\leq
		\sup_{\R^N\setminus\Omega} u + \frac{2^{N+5}\sqrt{N}\,\diam\Omega}{\beta}\,\|f^+\|_{L^N(K_u)}+o(\eps^0),
\end{equation*}
which is the classical ABP estimate plus an error term vanishing when $\eps\to 0$.

Observe that the error depends on $f$, moreover it does not vanish uniformly on $f$.
Also observe that the ABP estimate requires $f$ to be continuous.
The standard version of the ABP in the context of PDEs is with $L^N$-norm of $f$ in the right hand side.
Unfortunately such an inequality does not hold in our setting.
That is, for a general $f$, an inequality such as
\[
\E^x\left[\eps^2\sum_{i=0}^{\tau-1}f(x_i)\right]
\leq C\|f\|_N
\]
does not hold as the next example shows. 

\begin{example}
\label{ex:LN-abp-false}
Let us consider $\Om=B_2$, $\eps=1$, and  $f=1_{\Q^N}$, for which 
$$
\|f\|_N=0.
$$
Let $\nu$ be given by 
$$
\nu_x=\frac{\delta_{v_x}+\delta_{-v_x}}{2},
$$
where $v_x$ is such that $x+v_x\in\Q^N$. 
It follows that $\E^x\left[f(x_i)\right]\geq \frac{\alpha}{2}$ for any $x$.
Then since $\E^0\left[\eps^2\tau\right]\geq c$ we have 
$$
\E^0\left[\eps^2\sum_{i=0}^{\tau-1}f(x_i)\right]\geq \frac{c\alpha}{2}.
$$
\end{example}

We overcome the difficulty of the previous example in the following theorem, where we obtain a weaker version of the result.
Fortunately it is enough for our purposes, see Lemma~\ref{first}.

\begin{theorem}[$\eps$-ABP estimate with measurable $f$]
\label{boundomega}
Given $f:\Omega\to \R$ a non negative bounded measurable function, there exist $C>0$ (depending only on $N$) such that
\[
\E^x\left[\eps^2\sum_{i=0}^{\tau-1}f(X_i)\right]
\leq \left(\eps^2+\alpha \E^x[\eps^2\tau]\right)\|f\|_\infty
+ C(\diam\Omega+\Lambda\eps)\|f\|_N.
\]
\end{theorem}

\begin{proof}
First we extend $f:\Omega\to\R$ outside $\Omega$ defining $f(x)=0$ for every $x\in\R^N\setminus\Omega$. Then we define $\tilde f:\R^N\to\R$ as the function given by
\[
\tilde f(x)=\vint_{B_{\eps}(x)}f(y)\,dy
\]
for every $x\in\R^N$, so $\tilde f$ is continuous in $\R^N$ and, in particular $\tilde f\in C(\overline\Omega)$.
For $i\geq 1$ we have
\begin{align}
\label{condexpbound}
\E^x\left[f(X_i)|\F_{i-1}\right](\omega)
&=\alpha  \int f(x_{i-1}+\eps z) \,d\nu_{x_{i-1}}(z)+\beta\vint_{B_\eps(x_{i-1})} f(y)\,dy\nonumber\\
&\leq \alpha\|f\|_\infty+\beta \tilde f(x_{i-1}).
\end{align}
Since $\E^x\left[f(X_i)\right]=\E^x\left[\E^x[f(X_i)|\F_{i-1}]\right]$ we obtain
\[
\begin{split}
\E^x\left[\eps^2\sum_{i=0}^{\tau-1}f(X_i)\right]
&=\E^x\left[\eps^2f(X_0)+\eps^2\sum_{i=1}^{\tau-1}\E^x[f(X_i)|\F_{i-1}]\right]\\
&\leq \eps^2\|f\|_\infty+\E^x\left[\eps^2\sum_{i=1}^{\tau-1}\big(\alpha\|f\|_\infty+\beta \tilde f(X_{i-1})\big)\right],\\
\end{split}
\]
where we used \eqref{condexpbound}.
Rearranging terms, we get
\[
\begin{split}
\E^x\left[\eps^2\sum_{i=0}^{\tau-1}f(X_i)\right]
&\leq\eps^2\|f\|_\infty+\alpha \E^x[\eps^2(\tau-1)]\|f\|_\infty+\beta \E^x\left[\eps^2\sum_{i=0}^{\tau-2}\tilde f(X_i)\right]\\
&\leq\left(\eps^2+\alpha \E^x[\eps^2\tau]\right)\|f\|_\infty+\beta \E^x\left[\eps^2\sum_{i=0}^{\tau-1}\tilde f(X_i)\right].
\end{split}
\]
Observe that since $\tilde f\in C(\overline\Omega)$ and $\tilde f\geq 0$ we can apply Corollary~\ref{boundomegafcont}.
We obtain
\begin{equation}\label{E tilde f estimate}
\E^x\left[\eps^2\sum_{i=0}^{\tau-1}\tilde f(X_i)\right]\leq \frac{2^{N+3}}{\beta}(\diam\Omega+\Lambda\eps)\bigg(\sum_{\mathcal Q_\eps(\Omega)}(\sup_Q \tilde f)^N\bigg)^{1/N}\eps.
\end{equation}

For any fixed $Q\in\mathcal Q_\eps(\Omega)$, let $x_0$ denote the center of $Q$, so $Q=Q_{\frac{\eps}{4\sqrt{N}}}(x_0)$. Since $\diam Q=\eps/4$, then $|x-x_0|<\eps/8$ for every $x\in\overline Q$. Then
\begin{equation*}
	Q
	\subset
	B_\eps(x)
	\subset
	B_{\eps+|x-x_0|}(x_0)
	\subset
	B_{9\eps/8}(x_0)
	\subset
	Q_{9\eps/4}(x_0)
	=
	9\sqrt{N}Q_{\frac{\eps}{4\sqrt{N}}}(x_0)
	=
	9\sqrt{N}Q.
\end{equation*}
Let $\ell=\ell(N)\in\N$ be the unique odd integer such that $\ell-2<9\sqrt{N}\leq\ell$. In consequence $Q\subset B_\eps(x)\subset\ell Q$ for every $x\in\overline Q$. Since $\tilde f$ is continuous in $\R^N$, there exists some $\overline x\in\overline Q$ such that $\tilde f(\overline x)=\sup_Q\tilde f$ and thus
\begin{equation*}
	(\sup_Q\tilde f)^N
	=
	\bigg(\vint_{B_\eps(\overline x)}f(y)\,dy\bigg)^N
	\leq
	\vint_{B_\eps(\overline x)}f(y)^N\,dy
	\leq
	\frac{1}{|B_\eps|}\int_{\ell Q}f(y)^N\,dy,
\end{equation*}
where in the first inequality we have recalled Jensen's inequality for convex functions. 
Moreover, since the cubes in $\mathcal{Q}_\eps(\Omega)$ form a grid, it turns out that $\overline{\ell Q}$ can be expressed as the union of the cubes $\overline{Q'}$ such that $Q'\in\mathcal Q_\eps(\Omega)$ and $Q'\subset\ell Q$. Since any particular $Q'\in Q_\eps(\Omega)$ belongs to  $\mathrm{card}\{Q\in\mathcal Q_\eps(\Omega)\,:\,Q'\subset\ell Q\}=\ell^N$ number of cubes $\ell Q$, we can estimate the overlap and get
\begin{equation*}
\begin{split}
	\sum_{Q\in\mathcal Q_\eps(\Omega)}(\sup_Q \tilde f)^N
	\leq
	~&
	\frac{1}{|B_\eps|}\sum_{Q\in\mathcal Q_\eps(\Omega)}\int_{\ell Q}f(y)^N\,dy
	\\
	=
	~&
	\frac{1}{|B_\eps|}\sum_{Q'\in\mathcal Q_\eps(\Omega)}\mathrm{card}\{Q\in\mathcal Q_\eps(\Omega)\,:\,Q'\subset\ell Q\}\int_{Q'}f(y)^N\,dy
	\\
	= 
	~&
	\frac{\ell^N}{|B_\eps|}\sum_{Q'\in\mathcal Q_\eps(\Omega)}\int_{Q'}f(y)^N\,dy
	\\
	=
	~&
	\frac{\ell^N}{|B_\eps|}\int_\Omega f(y)^N\,dy,
\end{split}
\end{equation*}
where the last equality comes from the fact that $f\equiv 0$ outside $\Omega$. Taking the $N$-th root we finally obtain
\begin{equation*}
\begin{split}
	\bigg(\sum_{\mathcal Q_\eps(\Omega)}(\sup_Q \tilde f)^N\bigg)^{1/N}
	\leq
	~&
	\frac{\ell}{|B_1|^{1/N}\eps}\bigg(\int_\Omega f(y)^N\,dy\bigg)^{1/N},
\end{split}
\end{equation*}
and the result follows after inserting this in \eqref{E tilde f estimate}.
\end{proof}

\section{De Giorgi oscillation lemma}
\label{sec-DeGiorgi}

The main goal of this section is to prove Lemma~\ref{DeGiorgi}, a version of the classical De Giorgi oscillation lemma. 

We follow Krylov-Safonov argument in \cite{krylovs79} and \cite{krylovs80}, see also \cite[Chapter V, Section 7]{bass98}. However, our case is partly discrete and $\eps$ sets a natural limit for the scale that can be used in the proofs. This causes considerable changes.  
The key result is Theorem~\ref{thm:main} where we prove that a set of positive measure is reached by the process with positive probability. This then implies De Giorgi oscillation lemma in a straightforward manner by using a level set as  the set of positive measure.

One of the key steps in this section is the use of an adapted version Calder\'on-Zygmund decomposition, Lemma~\ref{CZ}.
The main difference with the classical version is that we do not consider cubes of scale smaller than $\eps$. If we simply stop the decomposition once we reach the cubes of scale $\eps$, then we would lose control between the original set and union of cubes in the decomposition. Therefore we need a subtle additional condition in the decomposition for cubes of size $\eps$.

The first step in our argument is to prove that sets of `large density' are reached by the process with positive probability.
This is done in the following lemma.  There we employ the $\eps$-ABP estimate, Theorem~\ref{boundomega}, with the characteristic function of $A$ as a right hand side, and further estimates from Section~\ref{time estimates}.
Recall that $T_A$ denotes the hitting time for $A$ and $\tau_A$ the exit time, that is
\[
T_A=\min \{k\in\N:X_k\in A\} \quad \text{and} \quad \tau_A=\min \{k\in\N:X_k\not\in A\}.
\]
\begin{lemma}
\label{first}
There exists $\eps_1>0$, $\theta>0$ and $c>0$ such that if $0<\eps<\eps_1$, $x\in Q_{1/2}$, $A\subset Q_1$ and $|Q_1\setminus A|<\theta$ then,
\[
\P^x(T_A<\tau_{Q_1})\geq c.
\] 
\end{lemma}

\begin{proof}
We denote $\tau=\tau_{Q_1}$ and $A^c=Q_1\setminus A$.
We write
\begin{equation}
\label{maintau}
\begin{split}
\E^x[\eps^2\tau]
&=\E^x[\eps^2\tau 1_{\{T_A<\tau\}}]+\E^x[\eps^2\tau 1_{\{T_A\geq\tau\}}]\\
&\leq\E^x[(\eps^2\tau)^2]^{\frac{1}{2}}\P^x(T_A<\tau)^{\frac{1}{2}}+\E^x\left[\eps^2\sum_{i=0}^{\tau-1}1_{A^c}(X_i)\right],
\end{split}
\end{equation}
where we have used Cauchy-Schwarz inequality and that $\sum_{i=0}^{\tau-1}1_{A^c}(X_i)=\tau$ when $T_A\geq\tau$.
By Theorem~\ref{boundomega} we have
\[
\E^x\left[\eps^2\sum_{i=0}^{\tau-1}1_{A^c}(X_i)\right]\leq\eps^2+\alpha \E^x[\eps^2\tau]+ C|A^c|^{1/N}.
\]
Combining this inequality with \eqref{maintau} we obtain
\[
\E^x[\eps^2\tau]
\leq\E^x[(\eps^2\tau)^2]^{\frac{1}{2}}\P^x(T_A<\tau)^{\frac{1}{2}}+\eps^2+\alpha \E^x[\eps^2\tau]+C\theta^{1/N}
\]
and rearranging terms
\[
\beta\E^x[\eps^2\tau]
\leq\E^x[(\eps^2\tau)^2]^{\frac{1}{2}}\P^x(T_A<\tau)^{\frac{1}{2}}+\eps^2+C\theta^{1/N}.
\]
By Lemma~\ref{taubounds} (observe that $\dist(x,\R^N\setminus Q_1)\geq 1/4$) and Corollary~\ref{corotau2} there exists $c_1,c_2>0$ such that $c_1\leq \E^x[\eps^2\tau]$  and $\E^x[(\eps^2\tau)^2]^{\frac{1}{2}}\leq c_2$.
Therefore
\[
\begin{split}
\beta c_1&\leq c_2\P^x(T_A<\tau)^{\frac{1}{2}}+C\theta^{1/N}+\eps^2\\
\end{split}
\]
and the result follows for $\eps$ and $\theta$ small enough.
\end{proof}

In the following lemma we prove that sets of positive measure are reached by the process with positive probability when $\eps_0/2\leq \eps<\eps_0$.
By performing a scaling of the space and step size we later use the result for cubes of size comparable to $\eps$ in Theorem~\ref{thm:main}, see Remark~\ref{scaling}.

\begin{lemma}
\label{second}
Given $0<\eps_0\leq 1$, there exists $\gamma=\gamma(\eps_0)>0$ such that if $\eps_0/2\leq \eps<\eps_0$, $x\in Q_1$, $A\subset Q_1$, we have
\[
\P^x(T_A<\tau_{Q_{10\sqrt{N}}})\geq \gamma|A|.
\] 
\end{lemma}

\begin{proof}
We define $N_0=\left\lceil \frac{2\sqrt N}{\eps_0}\right\rceil+1$ and consider the event $E$ of the first $N_0$ movements to be uniformly distributed.
That is $E=\{c_{1}=\cdots=c_{N_0}=1\}$.
We have
\[
\P^x(T_A<\tau_{Q_{10\sqrt{N}}})\geq \P^x(T_A<\tau_{Q_{10\sqrt{N}}}\big|E)\P(E).
\]
Observe that $\P(E)=\beta^{N_0}$. 
If a uniform random step occurs, then the step size is at most $\eps$.
Hence, after $N_0$ uniform random steps the token is at a distance of at most
\[
N_0\eps\leq
N_0\eps_0<
\left(\frac{2\sqrt N}{\eps_0}+2\right)\eps_0\leq
2\sqrt N+2\leq
4\sqrt N
\]
from the starting point.
Therefore, we have that all the steps until time $N_0$ are inside of $Q_{10\sqrt{N}}$, we have $\P^x(T_A<\tau_{Q_{10\sqrt{N}}}\big|E)\geq \P^x(X_{N_0}\in A\big|E)$.
We have proved that
\[
\P^x(T_A<\tau_{Q_{10\sqrt{N}}})\geq  \P^x(X_{N_0}\in A\big|E)\beta^{N_0}.
\]
We consider $U_i$ a sequence of independent random variables uniformly distributed in $B_1$.
And we define $Y=\sum_{i=1}^{N_0}U_i$.
Let $f$ denote the density of $Y$, it is a radial decreasing function strictly positive in the ball of radius $N_0$.
Given $x_0=x\in Q_1$ and $y\in Q_1$ we can bound
\[
f_{X_{N_0}|E}(y)=
\frac{1}{\eps^n}f((y-x)/\eps)\geq
\frac{1}{\eps_0^n}f(\sqrt{N}/\eps)\geq
\frac{1}{\eps_0^n}f(2\sqrt{N}/\eps_0).
\]
Since $\frac{2\sqrt N}{\eps_0}<N_0$ we have that $f(2\sqrt{N}/\eps_0)>0$.

Finally we obtain
\[
\P^x(T_A<\tau_{Q_{10\sqrt{N}}})\geq  \P^x(X_{N_0}\in A\big|E)\beta^{N_0}\geq \frac{1}{\eps_0^n} f(2\sqrt{N}/\eps_0) |A| \beta^{N_0} .
\]
Therefore the result holds for $\gamma= \frac{1}{\eps_0^n} f(2\sqrt{N}/\eps_0) \beta^{N_0}$.
\end{proof}

\begin{remark}
\label{scaling}
Given a cube $Q$ there exists an affine transformation $h(x)=ax+b$ such that $h(Q)=Q_1$.
Given the process $X_k$ we can consider the process $h(X_k)$.
Observe that this new process is of the type that we are considering for $\tilde \eps=a\eps$ and the pushforward measure $\tilde\nu$ given by $\tilde\nu_x(A)=\nu_{h^{-1}(x)}(h^{-1}(A))$.
Then, results established for $Q_1$ such as Lemma~\ref{second} can be applied to cubes of any size. Moreover, if $\eps_0/2\leq \tilde \eps<\eps_0$, then the constant $\gamma$ only depends  on $\eps_0$.
\end{remark}

Now we state our version of the Calder\'on-Zygmund lemma. In the discrete setting, the $\eps$ sets a natural limit for the scale. To control the error when stopping the decomposition at the level $\eps$, we introduce an additional condition. When applying the decomposition, a careful choice of the parameters allows us to guarantee the two opposite goals: there are enough cubes in the decomposition and the share of $A$ in measure is still large enough.

First we introduce some notation.
We denote by $\mathcal D_\ell$ the family of dyadic open subcubes of $Q_1$ of generation $\ell\in\N$.
That is $\mathcal D_0=\{Q_1\}$, $\mathcal D_1$ is the family of $2^N$ dyadic open cubes obtained by dividing $Q_1$, and so on.
Given $\ell\in\N$ and $Q\in\mathcal D_\ell$ we define $\mathrm{pre}(Q)\in\mathcal D_{\ell-1}$ as the unique dyadic cube in $\mathcal D_{\ell-1}$ containing $Q$. 

Let $0<\widetilde\delta < \delta<1$ and $A\subset Q_1$ a measurable set, such that
\begin{align*}
|A|\leq \delta.
\end{align*}
Next we will construct a collection of (open) cubes $\mathcal Q_B$, containing subcubes from generations $\mathcal D_0,\mathcal D_1,\dots,\mathcal D_L$, and a set
$$
B:\,=\cup_{Q\in \mathcal Q_B}Q.
$$
 By the assumption we first observe
\begin{align*}
|Q_1 \cap A|\leq \delta \abs{Q_1}.
\end{align*} 
Then we split $Q_1$ into $2^N$ dyadic cubes $\mathcal D_1$. For those  dyadic cubes $Q\in \mathcal D_1$ that satisfy 
\begin{align}
\label{eq:treshold}
|A\cap Q|>\delta|Q|,
\end{align}
we select $\mathrm{pre}(Q)$ into $\mathcal Q_B$. 
 
For other dyadic cubes that do not satisfy \eqref{eq:treshold} and are not contained in any cube already included in $\mathcal Q_B$,
we keep splitting, and again repeat the selection according to \eqref{eq:treshold}.  We repeat splitting $L\in \mathbb N$ times. At the level $L$, in addition to the previous process, we also select those cubes $Q\in \mathcal D_L$ (not the predecessors) into $\mathcal Q_B$ for which 
\begin{align}
\label{eq:treshold2}
 \delta |Q|\ge |A\cap Q|>\tilde \delta |Q|.
\end{align}
Now the following lemma holds.

\begin{lemma}[Calder\'on-Zygmund]
\label{CZ}
Let $A\subset Q_1$, $0<\widetilde\delta< \delta<1$, $L\in\N$ and $B$ be as above.
It holds that 
\[
|A|\leq \delta |B| + \tilde\delta.
\]
\end{lemma}

\begin{proof}
Observe that for $\mathrm{pre}(Q)$ selected according to \eqref{eq:treshold} into $\mathcal Q_B$, it holds that
\begin{align*}
|A\cap \mathrm{pre}(Q)|\le \delta|\mathrm{pre}(Q)|
\end{align*}
since otherwise we would have stopped splitting already at the earlier round. Also for cubes $Q$ selected according to \eqref{eq:treshold2}  into $\mathcal Q_B$, it holds that  $|A\cap Q|\le \delta|Q|$. Summing up, for all the cubes $Q\in \mathcal Q_B$, it holds 
\begin{align}
\label{eq:meas-bound}
|A\cap Q|\le \delta|Q|.
\end{align}
Moreover, by construction, the cubes in $\mathcal Q_B$ are disjoint.

We define $\mathcal G_L$ as a family of cubes of $\mathcal D_L$ that covers $Q_1\setminus B$ a.e.
It immediately holds a.e.\
\[
A\subset Q_1=\bigcup_{Q\in\mathcal Q_B} Q  \cup \bigcup_{Q\in\mathcal G_L} Q.
\]
By this, using \eqref{eq:meas-bound}, as well as observing that $|A\cap Q|\leq \tilde\delta|Q|$  by \eqref{eq:treshold2} for every $Q\in \mathcal G_L$, we get
\[
\begin{split}
|A|
&=\sum_{Q\in\mathcal Q_B} |A\cap Q| + \sum_{Q\in\mathcal G_L} |A\cap Q|\\
&\leq\sum_{Q\in\mathcal Q_B} \delta|Q| + \sum_{Q\in\mathcal G_L} \tilde\delta|Q|\\
&\leq  \delta |B|+\tilde\delta .\qedhere
\end{split}
\]
\end{proof}

Before proceeding to the main result, we need to show that if the stochastic process starts in certain cube, it will reach any subcube in the next level of the dyadic decomposition with positive probability. 
We also need to show that for any starting point in $Q_1$, the process reaches $Q_{1/2}$ with positive probability.
We obtain these results as a corollary of the following lemma.

\begin{lemma}
\label{R1R2R3}
Given $0<R_1<R_2<R_3$, there exists $\eps_2=\eps_2(R_1,R_2,R_3)>0$ and $p=p(R_1,R_2,R_3)>0$ such that for $x\in B_{R_2}$ we have
\[
\P^x(T_{B_{R_1}}<\tau_{B_{R_3}})\geq p
\]
for $\eps<\eps_2$.
\end{lemma}

\begin{proof}
For $c>0$ we consider the radial increasing function $\varphi(x)=-|x|^{-c}$.
For $x\in B_{R_3}\setminus B_{R_1}$ we have
\[
\begin{split}
\alpha \vint_{B_\eps(x)} \varphi(y)\,dy &+\beta \int \frac{\varphi(x+\eps z)+\varphi(x-\eps z)}{2} \,d\nu_x(z)\\
&\leq \varphi(x)+\frac{\alpha \eps^2}{2(N+2)}\Delta \varphi(x)+\beta\eps^2 \Lambda^2 \sup_{z:|z|=1}\langle D^2\varphi(x)z,z\rangle+o(\eps^2),
\end{split}
\]
where we have used the second order Taylor's expansion for $\varphi$ in $B_{R_3+\Lambda\eps_2}\setminus B_{R_1-\Lambda\eps_2}$.
Observe that $\varphi$ is smooth in that region for $\eps_2$ small enough.

We consider $s=|x|$ and $\phi(s)=\varphi(x)$.
Recall that for a radial function the eigenvalues of $D^2\varphi$ are $\phi''(s)$ with multiplicity 1 and $\phi'(s)/s$ with multiplicity $N-1$.
We obtain
\begin{equation}
\label{boundphi}
\begin{split}
\alpha& \vint_{B_\eps(x)} \varphi(y)\,dy +\beta \int \frac{\varphi(x+\eps z)+\varphi(x-\eps z)}{2} \,d\nu_x(z)\\
&\leq \phi(s)+\eps^2\left(\left(\frac{\alpha (N-1) }{2(N+2)}+\beta \Lambda^2\right)\frac{\phi'(s)}{s}+\frac{\alpha}{2(N+2)} \phi''(s)\right)
+o(\eps^2).\\
\end{split}
\end{equation}

We have $\phi'(s)=c s^{-c-1}$ and $\phi''(s)=-c(c+1) s^{-c-2}$.
Therefore, the right hand side in \eqref{boundphi} is smaller than $\varphi(x)$ for every $x\in B_{R_3}\setminus B_{R_1}$ for $c$ large enough and $\eps_2$ small enough.
Hence, $\varphi(x_n)$ is a supermartingale.

If $q=\P^x(T_{R_1}<\tau_{R_3})$, we obtain
\[
\varphi(R_3)(1-q)+\varphi(R_1-\Lambda\eps_2)q\leq
\mathbb E^x[\varphi(x_{\tau_{B_{R_3}\setminus B_{R_1}}})]
\leq \varphi(x)\leq \varphi(R_2).
\]
Hence $q\geq\frac{\varphi(R_3)-\varphi(R_2)}{\varphi(R_3)-\varphi(R_1-\Lambda\eps_2)}>0$.
\end{proof}

\begin{figure}
\begin{tikzpicture}[scale=0.3]
\draw[->] (-7,0) -- (7,0) node[right] {};
\draw[->] (0,-7) -- (0,7) node[above] {};

\draw [](0.5,0.5) circle (0.5);
\draw [](0.5,0.5) circle ({3/sqrt(2)});
\draw [](0.5,0.5) circle ({3.5*sqrt(2)});

\draw (1,1) -- (1,-1) -- (-1,-1) -- (-1,1) -- (1,1);

\draw (3,3) -- (3,-3) -- (-3,-3) -- (-3,3) -- (3,3);

\end{tikzpicture}
\caption{The inclusions of balls and cubes defined in Corollary~\ref{third} in the case that $P=Q_1\cap\{x:x_i>0 \text{ for } i=1,\dots,N\}$.}
\label{fig:ballsandcubes}
\end{figure}
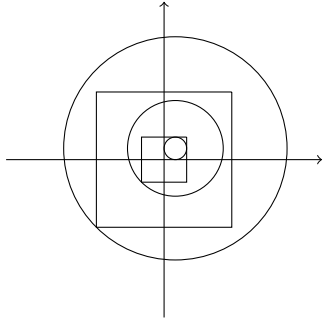

\begin{corollary}
\label{third}
There exists $\eps_2>0$ and $p>0$ such that if $\eps<\eps_2$, $x\in Q_1$ and $P\subset Q_1$ is a cube of side length 1/2, we have
\[
\P^x(T_P<\tau_{Q_{10\sqrt{N}}})\geq p.
\] 
\end{corollary}

\begin{proof}
For $y$ the center of $P$, $R_1=1/4, R_2=3\sqrt{N}/2$ and $R_3=\sqrt{N}(5\sqrt{N}+1/4)$, we have 
\[
B_{R_1}(y)\subset P\subset Q_1 \subset B_{R_2}(y) \subset Q_{10\sqrt{N}} \subset B_{R_3}(y)
\]
as shown in Figure~\ref{fig:ballsandcubes}.

Since $B_{R_1}(y)\subset P$ and $Q_{10\sqrt{N}} \subset B_{R_3}(y)$ we have
\[
\P^x(T_P<\tau_{Q_{10\sqrt{N}}})\geq \P^x(T_{B_{R_1}(y)}<\tau_{B_{R_3}(y)}).
\]
As $x\in Q_1\subset B_{R_2}(y)$, the result follows from Lemma~\ref{R1R2R3}.
\end{proof}

We are ready to prove our main result,
that a set of positive measure is reached by the process with positive probability. The idea is the following: given a suitable set $A$ we construct $B$ using the Calder\'on-Zygmund lemma  such that $B$ is larger than $A$ in measure. Using this, we prove that the process reaches $B$ (estimate \eqref{eq:key-B} below) and then $A$ by considering two alternatives (estimates \eqref{eq:key-A1} and \eqref{eq:key-A2} below).

\begin{theorem}
\label{thm:main}
There exists $\eps_0>0$ and a non decreasing function $\varphi:(0,1)\to (0,1)$ such that for every $\eps<\eps_0$, $A\subset Q_1$, $|A|>0$ and $x\in Q_1$, we have
\[
\P^x(T_A<\tau_{Q_{10\sqrt{N}}})>\varphi(|A|).
\] 
\end{theorem}

\begin{proof}
We define $\eps_0=\min\{\eps_1,\eps_2,1\}$ 
where $\eps_1$ and $\eps_2$ are given by Lemma~\ref{first} and Corollary~\ref{third}, respectively.
We also define
\[
\varphi(\xi)=\inf\left\{\P^x(T_A<\tau_{Q_{10\sqrt{N}}}):\nu,\eps<\eps_0, A\subset Q_1, |A|\geq \xi \text{ and } x\in Q_1\right\}
\]
for $\xi\in(0,1)$.
Observe that $\varphi(\xi)$ is non decreasing since for a larger $\xi$ the set where the infimum is taken is smaller.
We set
\begin{align}
\label{eq:q0}
q_0:\,=\inf\{ \xi\in (0,1) \,:\, \varphi(\xi)>0\}.
\end{align}
Since we want to prove that $q_0=0$,
we suppose, thriving for a contradiction, that $q_0>0$.

First, observe that 
\[
\P^x(T_A<\tau_{Q_{10\sqrt{N}}})\geq \P^x(T_{Q_{\frac{1}{2}}}<\tau_{Q_{10\sqrt{N}}}) \inf_{y\in Q_{1/2}} \P^y(T_A<\tau_{Q_1}).
\]
By Corollary~\ref{third} we have that $\P^x(T_{Q_{\frac{1}{2}}}<\tau_{Q_{10\sqrt{N}}})$ is positive and by Lemma~\ref{first} we have a positive lower bound for $\P^y(T_A<\tau_{Q_1})$ for $y\in Q_{1/2}$ whenever $ |Q_1\setminus A|$ is small enough.
Therefore the probability
$
\P^x(T_A<\tau_{Q_{10\sqrt{N}}})
$ 
is uniformly bounded from below for $A$ such that $|Q_1\setminus A|$ is small enough. We get $1>q_0$.

By the previous observation, we may choose $q>q_0$ such that $(q+q^2)/2<q_0$.
Thus for 
$\eta:\,=(q-q^2)/2$ we have 
$$
q-\eta=\frac{q+q^2}{2}<q_0<q.
$$
Given $A\subset Q_1$ with $q\geq |A|>q-\eta$, we consider the union of cubes $B$ constructed in Lemma~\ref{CZ} for $\delta=q$, $\tilde{\delta}=\eta$, and $L\in \N$ such that $2^L\eps<\eps_0\le 2^{L+1}\eps$.
Observe that $L$ depends on $\eps$, that is, the depth of the Calder\'on-Zygmund decomposition depends on $\eps$.
This is what allows us to have the smaller cubes in the decomposition of side length comparable to $\eps$.
All the other constants are independent of $\eps$.
With these choices, by the Calder\'on-Zygmund Lemma~\ref{CZ}, we have $|A|\leq q|B|+\eta$, that is
\[
|B|\geq \frac{|A|-\eta}{q}
\geq \frac{q-2\eta}{q}
= \frac{q-(q-q^2)}{q}=q.
\]
Hence, by the definition of $\varphi$ and \eqref{eq:q0}, we have 
\begin{align}
\label{eq:key-B}
\P^x(T_B<\tau_{Q_{10\sqrt{N}}})\geq \varphi(q)>0,
\end{align}
since by the choice of $q$ we had $q>q_0$.
We can estimate
\begin{align}
\label{eq:key-joint}
\P^x(T_A<\tau_{Q_{10\sqrt{N}}})\geq
\P^x(T_B<\tau_{Q_{10\sqrt{N}}})
\inf_{y\in B} \P^y(T_A<\tau_{Q_{10\sqrt{N}}}).
\end{align}
Now we estimate $\P^y(T_A<\tau_{Q_{10\sqrt{N}}})$ separating in two cases depending on $y\in B$.
Because of the construction of $B$ we know that one of the following must hold:
\begin{itemize}
\item
There exists a dyadic cube $Q$ with side length equal to $1/2^L$ such that $y\in Q\subset B$ and $|A\cap Q|>\tilde \delta|Q|=\eta|Q|$, or 
\item 
There exists a dyadic cube $Q$ with side length larger than or equal to  $1/2^L$ such that $y\in \mathrm{pre}(Q)\subset B$ and $|A\cap Q|>\delta |Q|=q|Q|$.
\end{itemize}

In the first case by scaling the cube $Q$ to $Q_1$ (see Remark~\ref{scaling}) we obtain a process for $\tilde \eps$ with $\eps_0/2\leq \tilde\eps=2^L\eps\leq \eps_0$.
By applying Lemma~\ref{second} we obtain for $y\in Q\subset B$
\begin{align}
\label{eq:key-A1}
\P^y(T_A<\tau_{Q_{10\sqrt{N}}})\geq \eta\gamma(\eps_0).
\end{align}
Observe that $\gamma$ depends on $\eps_0$ but not on $\eps$.

In the second case we scale $\mathrm{pre}(Q)$ to $Q_1$ and obtain a version of the process for $\tilde \eps\leq\eps_0$ and some $\tilde \nu$.
We may assume that the scaled version of $Q$ is  $P=Q_1\cap\{x:x_i>0 \text{ for } i=1,\dots,N\}$.
We can bound the probability of reaching $P$ by Corollary~\ref{third} and then the probability of reaching $A$ using that $|A\cap Q|>q|Q|$.
By the choice of $q$, we obtain for $y\in \mathrm{pre}(Q)\subset B$
\begin{align}
\label{eq:key-A2}
\P^y(T_A<\tau_{Q_{10\sqrt{N}}})\geq p\varphi(q)>0.
\end{align}
Using \eqref{eq:key-A1}, \eqref{eq:key-A2} and  \eqref{eq:key-B} in  \eqref{eq:key-joint},  we conclude that
\begin{align*}
\P^y(T_A<\tau_{Q_{10\sqrt{N}}})\geq
\varphi(q)\min\{\eta\gamma,p\varphi(q)\}>0.
\end{align*}
Hence, $\varphi(\xi)>0$ for every $\xi>q-\eta$, which is a contradiction. 
\end{proof}

Now we state a version of the classical De Giorgi oscillation lemma for the subsolutions of the DPP.

\begin{lemma}[De Giorgi oscillation lemma]
\label{DeGiorgi}
There exist $k>1$ and $C,\eps_0>0$ such that
for every $R>0$ and $\eps<\eps_0R$, if $u$ is a subsolution  to the DPP in $B_{kR}$ with $u\leq M$ in $B_{kR}$ and
\[
|B_{R}\cap \{u\leq m\}|\geq \theta |B_R|,
\]
for some $\theta>0$ and $m,M\in\R$, then there exist $\eta=\eta(\theta)>0$ such that
\[
\sup_{B_R} u \leq (1-\eta)M+\eta m+ C R^2\|f\|_\infty .
\]
\end{lemma}

\begin{proof}
Recall that $u$ is a subsolution, that is
\[
u (x)\leq\alpha  \int u(x+\eps z) \,d\nu_x(z)+\beta\vint_{B_\eps(x)} u(y)\,dy+\eps^2 f(x).
\]
We define $\tilde u(x)=u(2Rx)$, we have
\[
\tilde u \left(\frac{x}{2R}\right)\leq\alpha  \int \tilde u\left(\frac{x+ \eps z}{2R}\right) \,d \nu_x(z)+\beta\vint_{B_{\eps}(x)} \tilde u\left(\frac{y}{2R}\right)\,dy+ \eps^2 f(x).
\]
We consider $\tilde x=\frac{x}{2R}$ and define $\tilde\eps=\frac{\eps}{2R}$, $\tilde \nu$ such that $\nu_x=\tilde \nu_{\tilde x}$ and $\tilde f$ such that $\eps^2 f(x)=\tilde \eps^2 \tilde f(\tilde x)$, we get
\[
\tilde u \left(\tilde x\right)\leq\alpha  \int \tilde u\left(\tilde x+ \tilde \eps z\right) \,d \tilde \nu_{\tilde x}(z)+\beta\vint_{B_{\tilde \eps}(\tilde x)} \tilde u\left(y\right)\,dy+ \tilde\eps^2 \tilde f(\tilde x).
\]
That is, $\tilde u$ is a subsolution to the DPP in $B_{k/2}$ for $\tilde\eps$, $\tilde \nu$ and $\tilde f$ as defined above.
We consider the value of $\eps_0$  given by Theorem~\ref{thm:main}.
Observe that $\tilde\eps=\frac{\eps}{2R}<\eps_0$.
Also observe that
$\tilde u\leq M$ in $B_{k/2}$ and
\[
|B_{1/2}\cap \{\tilde u\leq m\}|\geq \theta |B_{1/2}|.
\]

We have $B_{1/2}\subset Q_1$.
We take $k=2(5N+\Lambda\eps_0)$, such that $X_{\tau_{Q_{10\sqrt{N}}}}\in B_{k/2}$.
We define $A=B_{1/2}\cap \{\tilde u\leq m\}$ and consider the stopping time 
\[
T=\min\{T_A,\tau_{Q_{10\sqrt{N}}}\}.
\]
For every $\tilde x\in Q_1$, we have
\[
\begin{split}
\tilde u(\tilde x)&\leq \E^{\tilde x}\left[\tilde  u(X_T)+\tilde\eps^2\sum_{i=0}^{T-1}\tilde  f(X_i)\right]\\
&\leq \E^{\tilde x}[\tilde u(X_T)|T_A<\tau_{Q_{10\sqrt{N}}}] \P^{\tilde x}(T_A<\tau_{Q_{10\sqrt{N}}})\\
&\quad +\E^{\tilde x}[\tilde u(X_T)|T_A>\tau_{Q_{10\sqrt{N}}}] (1-\P^{\tilde x}(T_A<\tau_{Q_{10\sqrt{N}}}))+\|\tilde f\|_\infty\E^{\tilde x}\left[\tilde\eps^2T\right]\\
&\leq M \P^{\tilde x}(T_A<\tau_{Q_{10\sqrt{N}}})+m(1-\P^{\tilde x}(T_A<\tau_{Q_{10\sqrt{N}}}))+ C\|\tilde f\|_\infty ,
\end{split}
\]
where the first inequality holds since $\tilde u$ is a subsolution to the DPP and we have bounded $\E^{\tilde x}\left[\tilde\eps^2T\right]$ by Lemma~\ref{taubounds}.

Observe that $\inf_{\tilde x\in B_{1/2}}\P^{\tilde x}(T_A<\tau_{Q_{10\sqrt{N}}})$ is positive as stated in Theorem~\ref{thm:main}.
Also observe that $\|\tilde f\|_\infty= (2R)^2 \|f\|_\infty$.
Therefore, we have proved the result since bounding $\tilde u(\tilde x)$ for every $\tilde x\in B_{1/2}\subset Q_1$ is equivalent to bound $u(x)$ for every $ x\in B_{R}$.
\end{proof}

Observe that the values of $k$ and $\eta$ do not depend on $\eps$ nor $R$.
And an analogous statement holds for supersolutions.

\section{Proof of the H\"older estimate}
\label{sec-Holder}

The H\"{o}lder estimate follows from the De Giorgi oscillation lemma, Lemma~\ref{DeGiorgi}, after a finite iteration.
We include here the details as we have to take special care of the role of $\eps$ in the arguments.

Given a function $u$ we define
\[
M(R)=\sup_{B_R} u, \quad  m(R)=\inf_{B_R} u \quad \text{and} \quad \O(R)=M(R)-m(R).
\]
We also define
\[
\osc(A)=\sup_{A} u-\inf_{A} u.
\]
Observe that $\osc(B_R)=\O(R)$.

\begin{lemma}
\label{osc}
There exists $\lambda<1$ and $k>1$ such that for every solution $u$ to the DPP defined in $B_{kR}$ we have
\[
\O(R)\leq \lambda \O(kR) + C R^2\|f\|_\infty
\]
for every $R$ and $\eps<R\eps_0$.
\end{lemma}

\begin{proof}
We can assume that $\O(kR)\neq 0$ where $k$ is given by Lemma~\ref{DeGiorgi}.
We consider $l = (M (kR) + m(kR))/2$. Either
\[
|\{u \geq l\} \cap B_R | \geq |B_R|/2,
\]
or
\[
|\{u \leq l\} \cap B_R | \geq |B_R|/2.
\]
Suppose that the first holds (the proof is completely analogous in the other case), then since $u\geq m(kR)$ and $l= m(kR)+\frac{\O(kR)}{2}$, Lemma~\ref{DeGiorgi} implies that
\[
m(R)\geq m(kR)+\eta\frac{\O(kR)}{2} - C R^2\|f\|_\infty
\]
for some $\eta=\eta(\frac{1}{2})>0$.
Then, since $M(R)\leq M(kR)$, we have
\[
\begin{split}
\O(R)&=M(R)-m(R)\\
&\leq M(kR)-m(kR)-\frac{\eta}{2}\O(kR)+ C R^2\|f\|_\infty\\
&= \left(1-\frac{\eta}{2}\right)\O(kR)+ C R^2\|f\|_\infty.
\end{split}
\]
Thus, the statement holds for $\lambda=1-\eta/2$.
\end{proof}

By iterating the oscillation estimate from Lemma~\ref{osc} we can obtain the H\"older regularity.
To that end we prove the following lemma.

\begin{lemma}
\label{rep}
If $\O(s)\geq 0$ is a non-decreasing function and $\O(s)\leq \lambda \O(ks)+ C s^2\|f\|_\infty$ for every $s>\xi$ for some $\lambda\in(0,1)$, $k>1$ and $\xi>0$ such that $\lambda k^2>1$, then
\[
\O(\rho)\leq \frac{1}{\lambda}\left(\frac{\rho}{R}\right)^\gamma ( \O(R)  + C R^2 \|f\|_\infty)
\]
for every $R\geq\rho>\xi$ where $\gamma=\frac{\log \frac{1}{\lambda}}{\log k}$.

\end{lemma}
\begin{proof}
Since $k>1$ there exist a unique $m\in\N_0$ such that
\[
k^m\leq \frac{R}{\rho}< k^{m+1}.
\]
By using repeatedly that $\O(s)\leq \lambda \O(ks)+ C s^2\|f\|_\infty$, for $s=\frac{R}{k^m},\frac{R}{k^{m-1}},\dots,\frac{R}{k}$ we obtain
\[
\begin{split}
\O(\rho)
&\leq \O\left(\frac{R}{k^m}\right)\\
&\leq \lambda \O\left(\frac{R}{k^{m-1}}\right)+C \left(\frac{R}{k^m}\right)^2\|f\|_\infty\\
&\leq \lambda^2 \O\left(\frac{R}{k^{m-2}}\right)+\lambda C \left(\frac{R}{k^{m-1}}\right)^2\|f\|_\infty+C \left(\frac{R}{k^m}\right)^2\|f\|_\infty\\
&\leq \lambda^m \O(R)+ C R^2 \|f\|_\infty \left(\frac{\lambda^{m-1} }{k^{2}}+\cdots+\frac{\lambda }{k^{2(m-1)}}+\frac{1}{k^{2m}}\right)\\
&= \lambda^m \O(R)+ C R^2 \|f\|_\infty\lambda^m \left(\frac{1}{\lambda k^{2}}+\cdots+\frac{1}{(\lambda k^2)^{m-1}}+\frac{1}{(\lambda k^2)^{m}}\right)\\
&\leq \lambda^m \O(R)+ C R^2 \|f\|_\infty\lambda^m \frac{1}{\lambda k^{2}-1}.\\
\end{split}
\]
Observe that we have used the hypothesis for values larger than $\xi$, in fact $$\frac{R}{k}\geq\frac{R}{k^{2}}\geq\dots\geq\frac{R}{k^m}\geq \rho>\xi.$$
We have,
\[
\begin{split}
\log \frac{R}{\rho}&< (m+1)\log k\\
\frac{\log \frac{R}{\rho}}{\log k}&< (m+1)\\
\lambda^\frac{\log \frac{R}{\rho}}{\log k}&>\lambda^{m+1}.\\
\end{split}
\]
Thus, the inequality follows since
\[
\lambda^m
=\frac{1}{\lambda}\lambda^{m+1} 
< \frac{1}{\lambda}\lambda^\frac{\log \frac{R}{\rho}}{\log k} 
= \frac{1}{\lambda}\left(\frac{\rho}{R}\right)^\gamma.\qedhere
\]
\end{proof}

Observe that Lemmas~\ref{osc} and~\ref{rep} prove that given $u$ a solution to the DPP defined in $B_R(x)$,
\[
	\osc(B_\rho(x))\leq \frac{1}{\lambda}\left(\frac{\rho}{R}\right)^\gamma (\osc(B_R(x))+ C R^2 \|f\|_\infty)
\]
for $\eps<\rho\eps_0$.
We are ready to prove the H\"older estimate.

\begin{proof}
[Proof of Theorem~\ref{Holder}]
Given $x, z\in B_R$ we consider $\rho=|x-z|$.

If $\rho\geq R$, we have
\[
|u(x)-u(z)|\leq 2\sup_{B_{2R}}|u|\frac{|x-z|^\gamma}{R^\gamma}.
\]

If $\rho< R$ and $\eps<\rho\eps_0$, we employ the previous lemma and obtain
\[
\begin{split}
|u(x)-u(z)|
&\leq \osc(\ol{B_\rho(x)})\\
&\leq \frac{1}{\lambda}\left(\frac{\rho}{R}\right)^\gamma (\osc(\ol{B_R(x)}))+ C R^2 \|f\|_\infty)\\
&\leq \frac{1}{\lambda}\left(\frac{\rho}{R}\right)^\gamma (\osc(B_{2R})+ C R^2 \|f\|_\infty)\\
&\leq \frac{2\sup_{B_{2R}}|u|+ C R^2 \|f\|_\infty}{R^\gamma\lambda}|x-z|^\gamma .
\end{split}
\]

In the case $\eps\geq \rho\eps_0$ then we can estimate $\osc(\ol{B_\rho(x)})$ by $\osc(\ol{B_{\rho'}(x)})$ for $\rho'=\frac{\eps}{\eps_0}$, we get
\[
\begin{split}
|u(x)-u(z)|
&\leq \osc(\ol{B_\rho(x)})\\
&\leq \osc(\ol{B_{\rho'}(x)})\\
&\leq \frac{1}{\lambda}\left(\frac{\rho'}{R}\right)^\gamma (\osc(\ol{B_R(x)})+ C R^2 \|f\|_\infty)\\
&\leq \frac{1}{\lambda}\left(\frac{\rho'}{R}\right)^\gamma (\osc(B_{2R})+ C R^2 \|f\|_\infty)\\
&\leq \frac{1}{\lambda}\left(\frac{\eps}{\eps_0R}\right)^\gamma 2\left(\sup_{B_{2R}}|u|+ C R^2 \|f\|_\infty\right).
\end{split}
\]
Observe that we have $\rho'\leq R$ since $\eps<\eps_0R$.
\end{proof}

\section{Generalization to Pucci-type operators and inequalities}
\label{sec:pucci}

Here we explain how to modify our arguments to include solutions to Pucci-type operators and inequalities.
Our method is robust and essentially the same arguments remain valid.

We start by defining the operator and then a stochastic process associated to it.

\begin{definition}
Let $u:\R^N\to\R$ be a bounded Borel measurable function. We define the maximal Pucci-type operator
\begin{equation*}
\begin{split}
	\L_\eps^+ u(x)
	:\,=
	~&
	\frac{1}{2\eps^2}\bigg(\alpha \sup_{\nu\in \M(B_\Lambda)} \int \delta u(x,\eps z) \,d\nu(z)	+\beta\vint_{B_1} \delta u(x,\eps y)\,dy\bigg)
	\\
	=
	~&
	\frac{1}{2\eps^2}\bigg(\alpha \sup_{z\in B_\Lambda} \delta u(x,\eps z) +\beta\vint_{B_1} \delta u(x,\eps y)\,dy\bigg),
\end{split}
\end{equation*}
where $\delta u(x,\eps y)=u(x+\eps y)+u(x-\eps y)-2u(x)$ for every $y\in B_\Lambda$. $\L_\eps^-$ is defined analogously just replacing $\sup$ by $\inf$.
\end{definition}

For related operators, see \cite{brustadlm}.

For each $\eps>0$ we consider a stochastic process starting at $x_0\in\R^N$.
The process is driven by a controller.
Given the value of $x_k$, the next position of the process $x_{k+1}$ is determined as follows. 
A biased coin is tossed.
If we get heads (probability $\alpha$), the controller chooses $z\in B_\Lambda$ and we have $x_{k+1}= x_k \pm \eps z$, each with probability one half.
If we get tails (probability $\beta$), $x_{k+1}$ is distributed uniformly in the ball $B_{\eps}(x_k)$. 

To be more precisely, a strategy $S$ for the controller is a measurable function defined on the partial histories, that is
\[
S(x_0,x_1,\dots, x_k)=z\in B_\Lambda.
\]
Then the process is moved accordingly to this choice. 
That is, given $A\in\mathcal B$  and $c=0\text{ or }1$, we have the following transition probabilities
\[
\pi_S(x_0,(c_1,x_1),\dots,(c_k,x_k),\{c\}\times A)=
\begin{cases}
\alpha \frac{\delta_z+\delta_{-z}}{2} \left(\frac{A-x_k}{\eps}\right) &\text{if }c=0,\\
\beta \frac{|A\cap B_\eps(x_k)|}{|B_\eps|} &\text{if }c=1,
\end{cases}
\]
where $z=S(x_0,x_1,\dots, x_k)$.

For a fixed strategy we have a process as before.
The only difference is that now the measure $\nu$ may depend not only in $x$ but $S$ (this does not introduce any difference in our arguments).
Fixed $S$ we can consider $\E^{x_0}_S$ the corresponding expectation.
All the estimates obtained for $\E^{x_0}$ hold for $\E^{x_0}_S$ and are independent of $S$.

We consider a game where the controller is paid $g(x_\tau)$ at the end and therefore it is her goal to maximize that value, the expectation for her earnings is given by
$$
u(x_0)=\sup_S \mathbb{E}^{x_0}_S [g (x_\tau)],
$$ 
where $\mathbb{E}^{x_0}_S$ stands for the expectation with respect to the process and $S$ is the strategy adopted by the controller.
The function $u :\mathbb{R}^N \to \mathbb{R}$ satisfies the DPP given by
\[
u (x) =\alpha \sup_{z\in B_\Lambda} \frac{u(x+\eps z)+u(x-\eps z)}{2}+\beta\vint_{B_\eps(x)} u(y)\,dy
\]
for $x \in \Omega$, and $u(x) = g(x)$ for $x \not\in \Omega$.

We can consider a version of the game where whenever the token leaves a point $x_i$, the controller is paid $\eps^2 f(x_i)$.
In this case the expectation for her earnings is given by
\begin{equation}
\label{defuS}
u(x)=\sup_S \mathbb{E}^{x}_S \left[\eps^2\sum_{i=0}^{\tau-1}f(X_i)+g (X_\tau)\right].
\end{equation}
It turns out, as will be shown below,  that $u$ is the unique  bounded Borel measurable function that satisfies
\begin{equation}
\label{DPP-L+}
u (x) =\alpha  \sup_{z\in B_\Lambda} \frac{u(x+\eps z)+u(x-\eps z)}{2}+\beta\vint_{B_\eps(x)} u(y)\,dy+\eps^2 f(x)
\end{equation}
for $x \in \Omega$ and $u(x) = g(x)$ for $x \not\in \Omega$.
Or equivalently $\L_\eps^+ u+f=0$.

The existence of a solution to equation \eqref{DPP-L+} can been seen as before. 
Next we prove the equivalent to Theorem~\ref{DPPsol}.

\begin{theorem}
\label{DPPsol-L+}
The function $u$ given by \eqref{defuS}
is the unique bounded solution to equation \eqref{DPP-L+} with boundary values $g$.
\end{theorem}

\begin{proof}
Let $v$ be a solution to equation \eqref{DPP-L+}.
Given a strategy $S_0$ we have
\[
\begin{split}
\E^{x_0}_{S_0}&\left[v(X_{k+1})+\eps^2\sum_{i=0}^{k}f(X_i)\Big|\F_{k}\right](\omega)\\
&\leq \alpha \sup_{z\in B_\Lambda} \frac{v(x_k+\eps z)+v(x_k-\eps z)}{2}+\beta\vint_{B_\eps(x_k)} v(y)\,dy+\eps^2\sum_{i=0}^{k}f(x_i)\\
&=v(x_k)+\eps^2\sum_{i=0}^{k-1}f(x_i).
\end{split}
\]
Thus $\{v(X_k)+\eps^2\sum_{i=0}^{k-1}f(X_i)\}_k$ is a supermartingale.
Then, by Doob's stopping time theorem (recall that $v$ and $f$ are bounded) we have
\[
\begin{split}
v(x)
\geq \mathbb{E}^x_{S_0}\left[v(X_\tau)+\eps^2\sum_{i=0}^{\tau-1}f(X_i)\right].
\end{split}
\]
Since this holds for every strategy and $v(x_\tau)=u(x_\tau)$, we get
\[
\begin{split}
v(x)
\geq \sup_S \mathbb{E}^x_{S}\left[v(X_\tau)+\eps^2\sum_{i=0}^{\tau-1}f(X_i)\right]=u(x).
\end{split}
\]

On the other hand, given $\eta>0$ we consider a strategy $S_0$ that almost maximizes the right hand side of $\eqref{DPP-L+}$, that is 
$
S_0(x_0,\dots,x_k)=\tilde z\in B_\Lambda
$
such that 
\[
\frac{v(x_k+\eps \tilde z)+v(x_k-\eps \tilde z)}{2}
\geq
\sup_{z\in B_\Lambda} \frac{v(x_k+\eps z)+v(x_k-\eps z)}{2}-\eta 2^{-(k+1)}.
\]
The strategy can be taken measurable similarly to Lemma 3.1 in \cite{luirops14}.

We have 
\[
\begin{split}
\E^{x_0}_{S_0}\Bigg[v(X_{k+1})+&\eps^2\sum_{i=0}^{k}f(X_i)-\eta 2^{-(k+1)}\Big|\F_{k}\Bigg](\omega)\\
&\geq \alpha  \sup_{z\in B_\Lambda} \frac{v(x_k+\eps z)+v(x_k-\eps z)}{2}+\beta\vint_{B_\eps(x_k)} v(y)\,dy\\
&\quad\quad\quad\quad\quad\quad\quad\quad\quad+\eps^2\sum_{i=0}^{k}f(x_i)-\eta 2^{-(k+1)}-\eta 2^{-(k+1)}\\
&=v(x_k)+\eps^2\sum_{i=0}^{k-1}f(x_i)-\eta 2^{-k}.
\end{split}
\]

Thus $\{v(X_k)+\eps^2\sum_{i=0}^{k-1}f(X_i)-\eta 2^{-k}\}_k$ is a submartingale.
Then
\[
\begin{split}
v(x) -\eta 
&\leq 
\mathbb{E}^x_{S_0}\left[v(X_\tau)+\eps^2\sum_{i=0}^{\tau-1}f(X_i)-\eta 2^{-\tau}\right]\\
&\leq 
\sup_S \mathbb{E}^x_{S}\left[v(X_\tau)+\eps^2\sum_{i=0}^{\tau-1}f(X_i)\right]\\
&=u(x).
\end{split}
\]
Since this holds for every $\eta>0$ we conclude that $v\leq u$.
Thus $v=u$, we have proved that $u$ is a solution to \eqref{DPP-L+} and that every solution coincides with it, there is a unique solution. 
\end{proof}

Finally, we state again Theorem~\ref{HolderL+L-(intro)}, which is our main result of this section: one only needs Pucci-type inequalities in order to obtain the regularity result.
\begin{theorem}
\label{HolderL+L-}
Let $f$ be a bounded Borel function. There exists $\eps_0>0$ such that if $u$ satisfies 
\begin{equation}
\label{L+L-ineq}
\L_\eps^+ u\ge -\abs{f},\quad \L_\eps^- u\le  \abs{f} 
\end{equation}
in $B_{2R}$ where $\eps<\eps_0R$, there exist $\gamma>0$ and $C>0$ such that
\[
|u(x)-u(z)|\leq \frac{C}{R^\gamma}\left(\sup_{B_{2R}}|u|+R^2\|f\|_\infty\right)\Big(|x-z|^\gamma+\eps^\gamma\Big)
\]
for every $x, z\in B_R$.
\end{theorem}

\begin{remark}
Observe that the $\eps$-ABP estimate (Theorem~\ref{ABP estimate}), as well as all the results from Section~\ref{sec-ABP}, are valid if we consider the maximal Pucci-type operator $\L_\eps^+$ instead of $\L_\eps$. This is due to the fact that (similarly as in equations \eqref{Lu+f>0} and \eqref{LGamma+f>0}), if $u$ is a bounded Borel measurable function satisfying $\L_\eps^+u+f\geq 0$ in $\Omega$, then $\L_\eps^+\Gamma+f\geq 0$ in $K_u$, where $\Gamma$ is the concave envelope of $u$ and $K_u$ is the set of contact points defined in \eqref{contact-set}. Hence, using this together with the fact that the second differences satisfy $\delta\Gamma(x_0,y)\leq 0$ for each $x_0\in K_u$, we can estimate the $\alpha$-term in $\L_\eps^+\Gamma(x_0)$ so
\begin{equation*}
		\L^+_\eps\Gamma(x_0)
		\leq
		\frac{\beta}{2\eps^2}\vint_{B_\eps}\delta\Gamma(x_0,y)\,dy.
\end{equation*}
This is analogous to the inequality \eqref{Gamma-ineq} in the proof of Lemma~\ref{key lemma}, and it is indeed the cornerstone in all the estimates from Section~\ref{sec-ABP}.
\end{remark}

With the analogous results of Sections~\ref{preliminaries} and \ref{sec-ABP} for $\L_\eps^+$ in hand, those of Section~\ref{sec-DeGiorgi} follow. However, there is a key modification needed is in the analogous version of Lemma~\ref{DeGiorgi} where, after establishing some estimates related to the stochastic process, solutions to the DPP are considered.
Here we adapt our argument to functions satisfying \eqref{L+L-ineq}.

\begin{lemma}
There exist $k>1$ and $C,\eps_0>0$ such that
for every $R>0$ and $\eps<\eps_0R$, if $u$ satisfies $\L_\eps^+ u\ge -\abs{f}$ in $B_{kR}$ with $u\leq M$ in $B_{kR}$ and
\[
|B_{R}\cap \{u\leq m\}|\geq \theta |B_R|,
\]
for some $\theta>0$ and $m,M\in\R$, then there exist $\eta=\eta(\theta)>0$ such that
\[
\sup_{B_R} u \leq (1-\eta)M+\eta m+ C R^2\|f\|_\infty .
\]
\end{lemma}

\begin{proof}
The function $u$ satisfies $\L_\eps^+ u\ge -\abs{f}$, that is
\[
u (x)\leq\alpha\sup_{z\in B_\Lambda} \frac{u(x+\eps z)+u(x-\eps z)}{2}+\beta\vint_{B_\eps(x)} u(y)\,dy+\eps^2 |f(x)|.
\]
We define $\tilde u(x)=u(2Rx)$, we have
\[
\tilde u \left(\frac{x}{2R}\right)\leq    \alpha \sup_{z\in B_\Lambda} \frac{\tilde u(\frac{x+ \eps z}{2R})+\tilde u(\frac{x- \eps z}{2R})}{2} +\beta\vint_{B_{\eps}(x)} \tilde u\left(\frac{y}{2R}\right)\,dy+ \eps^2 |f(x)|.
\]
We consider $\tilde x=\frac{x}{2R}$ and define $\tilde\eps=\frac{\eps}{2R}$ and $\tilde f$ such that $\eps^2 f(x)=\tilde \eps^2 \tilde f(\tilde x)$, we get
\begin{equation}
\label{L+full}
\tilde u \left(\tilde x\right)\leq\alpha \sup_{z\in B_\Lambda} \frac{\tilde u\left(\tilde x+ \tilde \eps z\right)+\tilde u\left(\tilde x- \tilde \eps z\right)}{2} +\beta\vint_{B_{\tilde \eps}(\tilde x)} \tilde u\left(y\right)\,dy+ \tilde\eps^2  |\tilde f(\tilde x)|.
\end{equation}
Where $\tilde u$ is is defined in $B_{k/2}$.
We consider the value of $\eps_0$  given by Theorem~\ref{thm:main}.
Observe that $\tilde\eps=\frac{\eps}{2R}<\eps_0$.
Also observe that
$\tilde u\leq M$ in $B_{k/2}$ and
\[
|B_{1/2}\cap \{\tilde u\leq m\}|\geq \theta |B_{1/2}|.
\]

Given $\eta>0$ we consider the strategy $S_0$ the almost maximizes the right hand side of $\eqref{L+full}$, that is
$
S_0(x_0,\dots,x_k)=\tilde z\in B_\Lambda
$
such that
\[
\frac{u(x+\eps \tilde z)+u(x-\eps \tilde z)}{2}
\geq
\sup_{z\in B_\Lambda} \frac{u(x+\eps z)+u(x-\eps z)}{2}-\eta 2^{-(k+1)}.
\]
As in the proof of Theorem~\ref{DPPsol-L+} we get that $\{v(X_k)+\eps^2\sum_{i=0}^{k-1}f(X_i)-\eta 2^{-k}\}_k$ is a submartingale for $\E^{\tilde x}_{S_0}$.

We have $B_{1/2}\subset Q_1$.
We take $k=2(5N+\Lambda\eps_0)$, such that $X_{\tau_{Q_{10\sqrt{N}}}}\in B_{k/2}$.
We define $A=B_{1/2}\cap \{\tilde u\leq m\}$ and consider the stopping time 
\[
T=\min\{T_A,\tau_{Q_{10\sqrt{N}}}\}.
\]
For every $\tilde x\in Q_1$ (and in particular for the ones in $B_{1/2}$), we have
\[
\begin{split}
\tilde u(\tilde x)-\eta &\leq \E^{\tilde x}_{S_0}\left[\tilde  u(X_T)+\tilde\eps^2\sum_{i=0}^{T-1}|\tilde  f(X_i)|-\eta 2^{-T}\right]\\
&\leq \E^{\tilde x}_{S_0}[\tilde u(X_T)|T_A<\tau_{Q_{10\sqrt{N}}}] \P^{\tilde x}_{S_0}(T_A<\tau_{Q_{10\sqrt{N}}})\\
&\quad +\E^{\tilde x}_{S_0}[\tilde u(X_T)|T_A>\tau_{Q_{10\sqrt{N}}}] (1-\P^{\tilde x}(T_A<\tau_{Q_{10\sqrt{N}}}))+\|\tilde f\|_\infty\E^{\tilde x}_{S_0}\left[\tilde\eps^2T\right]\\
&\leq M \P^{\tilde x}_{S_0}(T_A<\tau_{Q_{10\sqrt{N}}})+m(1-\P^{\tilde x}_{S_0}(T_A<\tau_{Q_{10\sqrt{N}}}))+ C\|\tilde f\|_\infty ,
\end{split}
\]
where we have bounded $\E^{\tilde x}_{S_0}\left[\tilde\eps^2T\right]$ by Lemma~\ref{taubounds}.

Observe that $\inf_{\tilde x\in B_{1/2}}\P^{\tilde x}(T_A<\tau_{Q_{10\sqrt{N}}})$ is positive as stated in Theorem~\ref{thm:main}.
Also observe that $\|\tilde f\|_\infty= (2R)^2 \|f\|_\infty$.
Therefore, we have proved the result since bounding $\tilde u(\tilde x)$ for every $\tilde x\in B_{1/2}$ is equivalent to bound $u(x)$ for every $ x\in B_{R}$.
Finally since the inequality holds for every $\eta>0$ it holds without it.
\end{proof}

\begin{remark}
\label{aplications}
Given nonempty subsets $\M_x\subset \M(B_\Lambda)$ for each $x\in\R^N$ with suitable measurability requirements, we can consider solutions to the equation
\begin{equation*}
u (x) =\alpha  \sup_{\nu\in \M_x} \int u(x+\eps z)\,d\nu(z)+\beta\vint_{B_\eps(x)} u(y)\,dy+\eps^2 f(x).
\end{equation*}
Observe that such function would satisfy \eqref{L+L-ineq} and therefore would be in the hypothesis of Theorem~\ref{HolderL+L-}.

Our results also cover Tug-of-War games with noise. Indeed, the value functions satisfy
\begin{align}
\label{eq:tgw-dpp}
\frac{1}{2\eps^2}\left(\alpha\left( \sup_{B_\eps(x)} u + \inf_{B_ \eps(x) } u	-2u(x)\right)+\beta\vint_{B_1} \delta u(x,\eps y)\,dy\right)+f(x)=0.
\end{align}
Since 
\[
\sup_{B_\eps(x)} u + \inf_{B_\eps(x)} u \leq \sup_{z\in B_1}  u(x+\eps z) +u(x-\eps z)
\]
we have $0\leq f+ \L^+_\eps u$ and similarly   $0\ge f+ \L^-_\eps u$. Therefore solutions to \eqref{eq:tgw-dpp} satisfy \eqref{L+L-ineq}.
Moreover, we can consider solutions to the DPP associated to the normalized p(x)-Laplacian given by
\begin{align}
\label{eq:p(x)-dpp}
u(x)=\frac{\alpha(x)}{2}\left( \sup_{B_\eps(x)} u + \inf_{B_ \eps(x) } u\right)+\beta(x) \vint_{B_\eps(x)} u(z) dz + \eps^2 f(x).
\end{align}
Let $\beta^-:=\inf_{x\in\Omega}\beta(x)>0$. Then we observe that
\begin{align*}
0&\le f(x)+ \frac{1}{2\eps^2}\bigg(\alpha(x) \sup_{z\in B_\Lambda} \delta u(x,\eps z) +\beta(x)\vint_{B_1} \delta u(x,\eps y)\,dy\bigg)\\
&\le f(x)+ \frac{1}{2\eps^2}\bigg((1-\beta^-) \sup_{z\in B_\Lambda} \delta u(x,\eps z) +\beta^-\vint_{B_1} \delta u(x,\eps y)\,dy\bigg).
\end{align*}
Similarly 
\begin{align*}
0&\ge f(x)+\frac{1}{2\eps^2}\bigg(\alpha(x) \inf_{z\in B_\Lambda} \delta u(x,\eps z) +\beta(x)\vint_{B_1} \delta u(x,\eps y)\,dy\bigg)\\
&\ge f(x)+ \frac{1}{2\eps^2}\bigg((1-\beta^-) \inf_{z\in B_\Lambda} \delta u(x,\eps z) +\beta^-\vint_{B_1} \delta u(x,\eps y)\,dy\bigg).
\end{align*}
Thus solutions to \eqref{eq:p(x)-dpp} satisfy the hypotheses of Theorem~\ref{HolderL+L-}.

In a similar way there is a large family of discrete operators associated to different PDEs that are in the hypothesis of our main result.

\end{remark}

\ 

\noindent \textbf{Acknowledgements.}  \'A.~A. is partially supported by a UniGe starting grant ``curiosity driven'' and grants MTM2017-85666-P, 2017 SGR 395. B.~P. and M.~P. are partially supported by the Academy of Finland project \#298641.

\def\cprime{$'$} \def\cprime{$'$} \def\cprime{$'$}


\begin{thebibliography}{PSSW09}

\bibitem[AHP17]{arroyohp17}
{\'A}.~Arroyo, J.~Heino, and M.~Parviainen.
\newblock Tug-of-war games with varying probabilities and the normalized $p
  (x)$-{L}aplacian.
\newblock {\em Commun.\ Pure Appl.\ Anal.}, 16(3):915--944, 2017.

\bibitem[AP20]{arroyop}
{\'A}.~Arroyo and M.~Parviainen.
\newblock Asymptotic H\"older regularity for the ellipsoid process.
\newblock {\em ESAIM Control Optim. Calc. Var.}, 26:Paper No. 112, 31, 2020.

\bibitem[Bas98]{bass98}
R.~F. Bass.
\newblock {\em Diffusions and elliptic operators}.
\newblock Probability and its Applications (New York). Springer-Verlag, New
  York, 1998.

\bibitem[BER19]{blancer}
P.\ Blanc, C.\ Esteve, and J.~D Rossi.
\newblock The evolution problem associated with eigenvalues of the hessian.
\newblock {\em J. Lond. Math. Soc. (2)}, 102(3):1293--1317, 2020.

\bibitem[BPR17]{blancpr17}
P.\ Blanc, J.~P. Pinasco, and J.~D. Rossi.
\newblock Maximal operators for the {$p$}-{L}aplacian family.
\newblock {\em Pacific J. Math.}, 287(2):257--295, 2017.

\bibitem[BLM18]{brustadlm}
K.~K Brustad, P.~Lindqvist, and J.~J. Manfredi.
\newblock A discrete stochastic interpretation of the dominative $ p
  $-laplacian.
\newblock {\em Differential Integral Equations}, 33(9-10):465--488, 2020.

\bibitem[CC95]{caffarellic95}
L.\ Caffarelli and X.~Cabr{\'e}.
\newblock {\em Fully nonlinear elliptic equations}, volume~43 of {\em American
  Mathematical Society Colloquium Publications}.
\newblock American Mathematical Society, Providence, RI, 1995.

\bibitem[CLU14]{caffarellilu14}
L.~Caffarelli, R.~Leit{\~a}o, and J.~M. Urbano.
\newblock Regularity for anisotropic fully nonlinear integro-differential
  equations.
\newblock {\em Math. Ann.}, 360(3-4):681--714, 2014.

\bibitem[CS09]{caffarellis09}
L.\ Caffarelli and L.~Silvestre.
\newblock Regularity theory for fully nonlinear integro-differential equations.
\newblock {\em Comm. Pure Appl. Math.}, 62(5):597--638, 2009.

\bibitem[CTU20]{caffarellitu20}
L.\ Caffarelli, R.\ Teymurazyan, and J.~M. Urbano.
\newblock Fully nonlinear integro-differential equations with deforming
  kernels.
\newblock {\em Communications in Partial Differential Equations}, pages 1--25,
  2020.

\bibitem[CD12]{changlarad12}
H.~Chang Lara and G.~D\'avila.
\newblock Regularity for solutions of nonlocal, nonsymmetric equations.
\newblock {\em Ann.\ Inst.\ H.\ Poincar{\'e} Anal.\ Non Lin{\'e}aire},
  29(6):833--859, 2012.

\bibitem[GS12]{guillens12}
N.~Guillen and R.~W. Schwab.
\newblock Aleksandrov-{B}akelman-{P}ucci type estimates for
  integro-differential equations.
\newblock {\em Arch. Ration. Mech. Anal.}, 206(1):111--157, 2012.

\bibitem[GT01]{gilbargt01}
D.~Gilbarg and N.~S. Trudinger.
\newblock {\em Elliptic partial differential equations of second order}.
\newblock Classics in Mathematics. Springer-Verlag, Berlin, 2001.
\newblock Reprint of the 1998 edition.

\bibitem[Har16]{hartikainen16}
H.~Hartikainen.
\newblock A dynamic programming principle with continuous solutions related to
  the {$p$}-{L}aplacian, {$1 < p < \infty$}.
\newblock {\em Differential Integral Equations}, 29(5-6):583--600, 2016.

\bibitem[KS79]{krylovs79}
N.~V. Krylov and M.~V. Safonov.
\newblock An estimate for the probability of a diffusion process hitting a set
  of positive measure.
\newblock {\em Dokl. Akad. Nauk SSSR}, 245(1):18--20, 1979.

\bibitem[KS80]{krylovs80}
N.~V. Krylov and M.~V. Safonov.
\newblock A property of the solutions of parabolic equations with measurable
  coefficients.
\newblock {\em Izv. Akad. Nauk SSSR Ser. Mat.}, 44(1):161--175, 239, 1980.

\bibitem[KT90]{kuot90}
H.~J. Kuo and N.~S. Trudinger.
\newblock Linear elliptic difference inequalities with random coefficients.
\newblock {\em Math. Comp.}, 55(191):37--53, 1990.

\bibitem[Kus15]{kusuoka15}
S.~Kusuoka.
\newblock H\"older continuity and bounds for fundamental solutions to
  nondivergence form parabolic equations.
\newblock {\em Anal. PDE}, 8(1):1--32, 2015.

\bibitem[Law91]{lawler91}
G.~F. Lawler.
\newblock Estimates for differences and {H}arnack inequality for difference operators coming from random walks with symmetric, spatially inhomogeneous, increments.
\newblock {\em Proc. London Math. Soc.}, 63(3):552--568, 1991.

\bibitem[Lew20]{lewicka20}
M.~Lewicka.
\newblock {\em A Course on Tug-of-War Games with Random Noise}.
\newblock Universitext. Springer-Verlag, Berlin, 2020.
\newblock Introduction and Basic Constructions.

\bibitem[LP17]{luirop17}
H.~Luiro and M.~Parviainen.
\newblock Gradient walk and $ p $-harmonic functions.
\newblock {\em Proc.\ Amer.\ Math.\ Soc.}, 145(10):4313--4324, 2017.

\bibitem[LP18]{luirop18}
H.~Luiro and M.~Parviainen.
\newblock Regularity for nonlinear stochastic games.
\newblock {\em Ann.\ Inst.\ H.\ Poincar{\'e} Anal.\ Non Lin{\'e}aire},
  35(6):1435--1456, 2018.

\bibitem[LPS13]{luirops13}
H.\ Luiro, M.\ Parviainen, and E.~Saksman.
\newblock Harnack's inequality for $p$-harmonic functions via stochastic games.
\newblock {\em Comm.\ Partial Differential Equations}, 38(11):1985--2003, 2013.

\bibitem[LPS14]{luirops14}
H.~Luiro, M.~Parviainen, and E.~Saksman.
\newblock On the existence and uniqueness of $p$-harmonious functions.
\newblock {\em Differential and Integral Equations}, 27(3/4):201--216, 2014.

\bibitem[LR86]{lindvallr86}
T.~Lindvall and L.~C.~G. Rogers.
\newblock Coupling of multidimensional diffusions by reflection.
\newblock {\em Ann. Probab.}, 14(3):860--872, 1986.

\bibitem[MPR12]{manfredipr12}
J.J. Manfredi, M.~Parviainen, and J.D. Rossi.
\newblock On the definition and properties of p-harmonious functions.
\newblock {\em Ann. Scuola Norm. Sup. Pisa Cl. Sci.}, 11(2):215--241, 2012.

\bibitem[PS08]{peress08}
Y.~Peres and S.~Sheffield.
\newblock Tug-of-war with noise: a game-theoretic view of the
  {$p$}-{L}aplacian.
\newblock {\em Duke Math. J.}, 145(1):91--120, 2008.

\bibitem[PSSW09]{peresssw09}
Y.~Peres, O.~Schramm, S.~Sheffield, and D.~B. Wilson.
\newblock Tug-of-war and the infinity {L}aplacian.
\newblock {\em J. Amer. Math. Soc.}, 22(1):167--210, 2009.

\bibitem[PT76]{puccit76}
C.~Pucci and G.~Talenti.
\newblock Elliptic (second-order) partial differential equations with
  measurable coefficients and approximating integral equations.
\newblock {\em Advances in Math.}, 19(1):48--105, 1976.

\bibitem[Ruo16]{ruosteenoja16}
E.~Ruosteenoja.
\newblock Local regularity results for value functions of tug-of-war with noise
  and running payoff.
\newblock {\em Adv. Calc. Var.}, 9(1):1--17, 2016.

\bibitem[SS16]{schwabs16}
R.~W. Schwab and L.~Silvestre.
\newblock Regularity for parabolic integro-differential equations with very irregular kernels.
\newblock {\em Anal. PDE}, 9(3):727--772, 2016.

\end{thebibliography}
\end{document}